\documentclass[12pt,draft]{article} 
\usepackage{amsmath,amssymb,amsthm,amsfonts,bm}
\usepackage{enumerate,color}
\makeatletter
\def\@cite#1#2{[{{\bfseries #1}\if@tempswa , #2\fi}]}
\renewcommand{\section}{%
\@startsection{section}{1}{\z@}
{0.5truecm plus -1ex minus -.2ex}%
{1.0ex plus .2ex}{\bfseries\large}}
\def\@seccntformat#1{\csname the#1\endcsname.\ }
\makeatother

\numberwithin{equation}{section} 
\newtheorem{thm}{Theorem}[section]
\newtheorem{corollary}[thm]{Corollary}
\newtheorem{lem}[thm]{Lemma}

\theoremstyle{definition}
\newtheorem{df}{Definition}[section]
\newtheorem{remark}{Remark}[section]
\newtheorem*{prth1.1}{Proof of Theorem 1.1}

\newcommand{\ep}{\varepsilon}

\newcommand{\norm}[4]{\Vert #1 \Vert _{{#2}^{#3}(#4)}}

\newcommand{\h}{H}
\newcommand{\tmax}{T_{{\rm max},\ep}}
\newcommand{\lp}[2]{\|#2\|_{L^{#1}(\Omega)}}

\newcommand{\nep}{n_{\ep}}
\newcommand{\cep}{c_{\ep}}
\newcommand{\uep}{u_{\ep}}
\newcommand{\io}{\int_\Omega}

\newcommand{\cd}{(\cdot,t)}
\newcommand{\iio}{\int_0^T\!\! \io}
\newcommand{\iiio}{\int_0^{\widetilde{T}}\!\! \io}
\newcommand{\kl}[1]{\left( #1 \right)}
\newcommand{\ol}{\overline}

\newcommand{\abs}{\medskip}

\newcommand{\de}{D_{\e}}
\newcommand{\ye}{y_\e}
\newcommand{\ze}{z_\e}
\newcommand{\Ne}{\nep}
\newcommand{\Ce}{\cep}
\newcommand{\Ue}{\uep}

\newcommand{\tilt}{\widetilde{T}}
\newcommand{\deltamu}{\delta_{\mu,0}}

\def\D{\Delta}
\def\om{\Omega}
\def\na{\nabla}
\def\vp{\varphi}
\def\e{\varepsilon}
\def\h{\hspace}
\def\etmax{T_{{\rm max},\e}}

\def\norm#1{\|#1\|}
 
\begin{document}

\footnote[0]
    {2010{\it Mathematics Subject Classification}\/. 
    Primary: 35K55; Secondary: 92C17; 35Q35.
    }
\begin{center}
    \Large{{\bf 
   How strongly does diffusion or logistic-type degradation affect existence of global weak solutions in a chemotaxis-Navier--Stokes system? 
              }}
\end{center}
\vspace{5pt}
\begin{center}
     Masaaki Mizukami\footnote{Partially supported by 
    JSPS Research Fellowships 
    for Young Scientists (No.\ 17J00101).}\\
    \vspace{2pt}
    Department of Mathematics, 
    Tokyo University of Science\\
    1-3, Kagurazaka, Shinjuku-ku, Tokyo 162-8601, Japan\\
    {\tt masaaki.mizukami.math@gmail.com}\\
\end{center}

\vspace{2pt}
\newenvironment{summary}
{\vspace{.5\baselineskip}\begin{list}{}{%
     \setlength{\baselineskip}{0.85\baselineskip}
     \setlength{\topsep}{0pt}
     \setlength{\leftmargin}{12mm}
     \setlength{\rightmargin}{12mm}
     \setlength{\listparindent}{0mm}
     \setlength{\itemindent}{\listparindent}
     \setlength{\parsep}{0pt} 
     \item\relax}}{\end{list}\vspace{.5\baselineskip}}
\begin{summary}
{\footnotesize {\bf Abstract.}
This paper considers the chemotaxis-Navier--Stokes system 
with nonlinear diffusion and logistic-type degradation term   
\begin{equation*}
     \begin{cases}
         n_t + u\cdot\nabla n =
           \na\cdot(D(n)\na n) - \nabla\cdot(n \chi(c) \nabla c) + \kappa n - \mu n^\alpha, 
         &x\in \Omega,\ t>0,
 \\[2mm]
         c_t + u\cdot\nabla c = 
         \Delta c - nf(c),
          &x \in \Omega,\ t>0,
 \\[2mm]
        u_t + (u\cdot\nabla)u 
          = \Delta u + \nabla P + n\nabla\Phi + g, 
            \quad \nabla\cdot u = 0, 
          &x \in \Omega,\ t>0, 
     \end{cases} 
 \end{equation*} 
where $\Omega\subset \mathbb{R}^3$ is a bounded smooth domain; 
$D \ge 0$ is a given smooth function such that $D_1 s^{m-1} \le D(s) \le D_2 s^{m-1}$ for all $s\ge 0$ 
with some $D_2 \ge D_1 > 0$ and some $m > 0$; 
$\chi,f$ are given smooth functions satisfying 
\begin{align*}
\kl{\frac f\chi}' >0, \quad \kl{\frac f\chi}'' \le 0, \quad (\chi f)' \ge 0 \quad \mbox{on} \ [0,\infty); 
\end{align*}
$\kappa \in \mathbb{R},\mu \ge0,\alpha>1$ are constants. 
This paper shows existence of global weak solutions to the above system under the condition that 
\begin{align*} 
m >\frac{2}{3},\quad \mu \ge 0 \quad  \mbox{and}\quad  \alpha >1 
\end{align*}
hold, or that 
\begin{align*} 
m> 0, \quad \mu>0 \quad  \mbox{and} \quad \alpha > \frac{4}{3} 
\end{align*}
hold.  This result asserts that ``strong'' diffusion effect or ``strong'' logistic damping derives existence of global weak solutions even though the other effect is ``weak'', and can include previous works \cite{KM,Lankeit_2016,Winkler_2016,zhang_li}.
}
\end{summary}
\vspace{10pt}

\newpage

\section{Introduction}
This work deals with the chemotaxis-Navier--Stokes system with nonlinear diffusion and logistic-type degradation term
\begin{align}\label{Intro;Pro1}
 \begin{cases}
  n_t + u\cdot \na n = \Delta n^m - \chi \na\cdot (n \na c) + 
  \kappa n - \mu n^\alpha, &
 \\ 
  c_t + u \cdot \na c  = \Delta c - nc, &
 \\
  u_t + (u\cdot \na)u = \Delta u + \na P + n\na \Phi
 \end{cases}
\end{align}
for $x\in \Omega$ and $t>0$, 
where $\Omega \subset \mathbb{R}^3$ is a bounded domain with smooth boundary, $m,\chi>0$, $\kappa,\mu\ge 0$ and $\alpha >1$ are constants and $\Phi$ is a given function, 
and consider the question: 
\begin{center}
{\it 
  How strongly does diffusion or logistic-type degradation affect existence of \\ 
  global weak solutions in a chemotaxis-Navier--Stokes system? 
}%
\end{center}
More precisely, the purpose of this work is to determine conditions for $m$ and $\alpha$ which derive global existence of weak solutions to the system \eqref{Intro;Pro1}. 
The system \eqref{Intro;Pro1} is a generalization of a chemotaxis-Nevier--Stokes system which is proposed by Tuval et al.\ \cite{Tuvaletal} and  describes the situation where a species in a drop of water moves towards higher concentration of oxygen according to a property called chemotaxis. 
Here chemotaxis is a property such that a species reacts on some chemical substance and moves towards or moves away from higher concentration of that substance. 
Chemotaxis is one of important properties in the animals' life, e.g., movement of sperm, migrations of neurons and lymphocytes and tumor invasion. 
In \eqref{Intro;Pro1}, $\mathbb{R}$-valued unknown functions $n=n(x,t)$, $c=c(x,t)$, $P=P(x,t)$ shows the density of species, the concentration of oxygen, the pressure of the fluid, respectively, and an $\mathbb{R}^3$-valued unknown function $u=u(x,t)$ describes the fluid velocity field. 

\abs
In the study of the system \eqref{Intro;Pro1}, we often refer to the study of the  chemotaxis system
\begin{align}\label{Intro;Pro2}
 \begin{cases}
  n_t = \Delta n^m - \chi \na \cdot (n\na c) + \kappa n - \mu n^\alpha, 
 & 
\\ 
  c_t = \Delta c - c + n.
 &
 \end{cases}
\end{align}
Thus we first introduce several known results about the chemotaxis system: 

\abs
In the study of the Keller--Segel system, that is, the system \eqref{Intro;Pro2} with $m=1$ and $\kappa=\mu=0$
\begin{align*}
\begin{cases} 
 n_t = \Delta n - \chi \na \cdot (n\na c), & \\ 
 c_t = \Delta c - c + n, & 
\end{cases}
\end{align*}
the chemotaxis term $- \chi \na \cdot (n\na c)$ derives blow-up phenomena in some cases; in the $2$-dimensional setting it was shown that there exists $\vartheta >0$ such that, if a mass of an initial data of $n$ is less than $\vartheta$ then global bounded classical solutions exist (see Nagai--Senba--Yoshida \cite{Nagai-Senba-Yoshida}), and for all $M > \vartheta$ there is an initial data $n_0$ of $n$ such that $M=\io n_0$ and a corresponding solution blows up in finite/infinite time (see Horstmann--Wang \cite{Horstmann-Wang} and Mizoguchi--Winkler \cite{Mizoguchi-Winkler}); in the $3$-dimensional setting 
for all $M>0$ there is an initial data $n_0$ of $n$ such that $M=\io n_0$ and a corresponding solution blows up in finite time; 
related works about the Keller--Segel system can be found in \cite{Xinru_higher,Osaki-Yagi,win_aggregationvs}; 
blow-up phenomena is excluded in the 1-dimensional case (\cite{Osaki-Yagi}); global existence results in the higher-dimensional setting are in \cite{win_aggregationvs,Xinru_higher}.  

\abs 
On the other hand, in the chemotaxis system with logistic term which is \eqref{Intro;Pro2} with $m=1$ and $\alpha =2$
\begin{align*}
\begin{cases}
  n_t = \Delta n - \chi \na \cdot (n\na c) + \kappa n -\mu n^2, &  
\\ 
  c_t = \Delta c - c + n, & 
\end{cases}
\end{align*} 
the logistic term $\kappa n - \mu n^2$ suppresses blow-up phenomena; 
in the $2$-dimensional setting Osaki et al.\ \cite{OTYM} and 
Jin--Xiang \cite{Tian-HaiYang_2018_2D} derived that for all $\mu>0$ there exist global classical solutions;  
Winkler \cite{Winkler_2010_logistic} showed global existence of classical solutions under some largeness condition for $\mu>0$; recently, Xiang \cite{Tian_2018_3D} obtained an explicit condition for $\mu>0$ to derive global existence of classical solutions; 
Lankeit \cite{lankeit_evsmoothness} established global existence of weak solutions for arbitrary $\mu>0$; 
more related works are in \cite{he_zheng,W2014}; 
Winkler \cite{W2014} and He--Zheng \cite{he_zheng} showed 
asymptotic behavior of global classical solutions. 

\abs 
However, the small logistic-type degradation damping may not suppress blow-up phenomena; in the parabolic--elliptic chemotaxis system with logistic-type degradation term 
\begin{align*}
\begin{cases} 
   n_t = \Delta n - \chi \na \cdot (n\na c) + \kappa n -\mu n^\alpha, & 
\\
  0 = \Delta c - c + n, & 
\end{cases} 
\end{align*}
Winkler \cite{Winkler_2018_bu_log} showed that, 
if $\alpha < \frac 76$ in the $3,4$-dimensional cases and if $\alpha < 1+\frac{1}{2(N-1)}$ in the $N$-dimensional case with $N\ge 5$,  
then there exists an initial data such that a corresponding solution blows up in finite time. 
Related works can be found in \cite{Viglialoro_2016,Viglialoro_2017,Viglialoro-Tomas,Winkler_2008}; 
Winkler \cite{Winkler_2008} showed existence of very weak solutions 
under the condition that $\alpha > 2-\frac 1N$; 
Viglialoro \cite{Viglialoro_2016,Viglialoro_2017} obtained existence of very weak solutions to the parabolic--parabolic system and their boundedness; 
Their large time behavior can be found in \cite{Viglialoro-Tomas}.  

\abs 
Moreover, in the chemotaxis system with degenerate diffusion 
\begin{align*}
\begin{cases}
   n_t = \Delta n^m - \chi \na \cdot (n^{q-1}\na c), & 
\\ 
  c_t = \Delta c - c + n & 
\end{cases} 
\end{align*}
with some $q\ge 2$, in the $N$-dimensional setting 
some smallness condition for $m\ge 1$ yields existence of blow-up solutions to the system; Ishida--Yokota \cite{Ishida-Yokota_2013} and Hashira--Ishida--Yokota \cite{Hashira-Ishida-Yokota} obtained that 
the condition that $m < q- \frac{2}{N}$ entails existence of an initial data such that a corresponding solution blows up in finite time; 
conversely, it was shown that the restriction of $m> q - \frac{2}{N} $ enables us to find global weak solutions (\cite{Ishida-Seki-Yokota}). 

\abs 
In summary, in the study of the chemotaxis system, 
some largeness condition for an effect of the logistic-type degradation or  
the nonlinear diffusion entails global existence, 
and some smallness condition for the effect derives existence of blow-up solutions. 
Does this happen also in the chemotaxis-Navier--Stokes system? 
In order to consider this question 
we next recall several related works about the chemotaxis-Navier--Stokes system \eqref{Intro;Pro1}: 

\abs
We first introduce results about the fluid-free system, which is \eqref{Intro;Pro1} with $u=0$; in the case that $m=1$ and $\kappa=\mu =0$ 
Tao--Winkler \cite{Tao-Winkler_2012_cons} obtained existence of global solutions and their large time behavior; Tao \cite{Tao_2011_cons} established global existence of bounded classical solutions to the system with $m=1$ and $\kappa=\mu=0$ under some smallness condition for an initial data of $c$; 
existence of global weak solutions to the system with $m>1$ and $\kappa=\mu=0$ in the two-dimensional setting is in \cite{Tao-Winkler_2012_KSF}; 
Winkler \cite{W-2015} obtained that  global bounded solutions of the system with $m>\frac 76$ and $\kappa=\mu=0$ exist 
in the three-dimensional setting, and Tao--Winkler \cite{Tao-Winkler_2013_cons} derived existence of locally bounded global solutions to 
the system with $m> \frac 87$ and $\kappa = \mu=0$; 
in the case that $m=1$, $\mu>0$ and $\alpha =2$ 
Zheng--Mu \cite{Zheng-Mu_2015} and Lankeit--Wang \cite{Lankeit-Wang_2017} showed global existence of bounded classical solutions 
under some smallness condition for an initial data of $c$; 
Jin \cite{Jin-2017} established existence of bounded global weak solutions to 
the system with $m>1$, $\mu>0$ and $\alpha =2$. 

\abs 
We then introduce results about the system \eqref{Intro;Pro1}. 
In the case that $\kappa= \mu =0$ and in the $3$-dimensional setting, 
it was shown that 
the diffusion effect dominate 
the chemotactic interaction; 
in the case that $m = 1$ Winkler \cite{Winkler_2016} showed global existence of weak solutions; 
in \eqref{Intro;Pro1} with $\kappa=\mu=0$, Zhang--Li  \cite{zhang_li} asserts that if $m\ge \frac 23$ then global weak solutions exist; however, there seem to be several miscalculations in the proof, 
e.g., in the proof of \cite[(3.6)]{zhang_li} they used the Young inequality 
\[
 a^{\frac 1{3m-1}} \le \ep a + C(\ep,m)
\]
with some $C(\ep,m) >0$  
for all $\ep>0$ and for $m\ge \frac 23$ 
even though this inequality does not hold when $m=\frac 23$ (which implies that $\frac 1{3m-1} = 1$); also in the case that $m>\frac 23$ there are still gaps in the proof  
 (for more details, see Remarks \ref{remark;zhang-li;young} and \ref{remark;zhang-li;esti} in this paper); 
although there are miscalculations, they constructed essential estimates for obtaining global existence of weak solutions;  
thus another purpose of this work is to correct arguments in \cite{zhang_li} and to establish some condition of $m$ for deriving global existence of weak solutions. 

\abs 
Moreover, in the case that $\mu>0$ and $\alpha=2$, 
it was established that, for all $\mu>0$, global weak solutions exist; 
Lankeit \cite{Lankeit_2016} first obtained global existence of weak solutions to \eqref{Intro;Pro1} with $m=1$ and $\alpha =2$; 
recently, global existence of weak solutions to \eqref{Intro;Pro1} with $m>0$ and $\alpha=2$ was shown in \cite{KM}; however, a general case such as that $m>0$ and $\alpha >1$ has not considered yet. 
Thus the main purpose of this paper is to obtain some conditions for $m >0$ and $\alpha >1$ which derive existence of global weak solutions to the chemotaxis-Navier--Stokes system. 
\abs 

In order to attain the purposes of this paper: 
\begin{itemize}
\setlength{\itemsep}{0cm}
\item to obtain some conditions for deriving global existence of weak solutions, 
\item to correct arguments in \cite{zhang_li} for establishing global weak solutions,  
\end{itemize}
we consider the following chemotaxis-Navier--Stokes system 
with nonlinear diffusion and logistic-type degradation term:
\begin{equation}\label{P}
     \begin{cases}
         n_t + u\cdot\nabla n = 
         \na \cdot (D(n)\na n) - \nabla\cdot(n\chi(c)\nabla c) 
          + \kappa n -\mu n^\alpha,
         &x\in \Omega,\ t>0,
 \\[2mm]
         c_t + u\cdot\nabla c = \Delta c -nf(c),
          &x \in \Omega,\ t>0,
 \\[2mm]
        u_t  + (u\cdot\nabla) u 
          = \Delta u + \nabla P + n\nabla\Phi  + g, 
            \quad \nabla\cdot u = 0, 
          &x \in \Omega,\ t>0, 
 \\[2mm]
        D(n)\partial_\nu n = \partial_\nu c = 0, \quad 
        u = 0, 
        &x \in \partial\Omega,\ t>0, 
 \\[2mm]
        n(x,0)=n_{0}(x),\ 
        c(x,0)=c_0(x),\ u(x,0)=u_0(x), 
        &x \in \Omega,
     \end{cases}
 \end{equation}
\noindent
where $\Omega$ is a bounded domain 
in $\mathbb{R}^3$ with smooth boundary $\partial\Omega$ and 
$\partial_\nu$ denotes differentiation with respect to the 
outward normal of $\partial\Omega$; 
$D$ is a function satisfying 
\begin{align}\label{condi;D}
 D\in C_{\rm loc}^{1+\gamma}([0,\infty)), \quad 
 D_1 s^{m-1} \le D(s) \le D_2 s^{m-1} 
 \quad \mbox{for all} \ s\ge 0 
\end{align}
with some $\gamma>0$, $D_2 \ge D_1 >0$ and $m>0$; $\kappa \in \mathbb{R}$, $\mu \ge 0$, $\alpha >1$ are constants; 
functions $\chi$ and $f$ are assumed that 
\begin{align}\label{condi;chi}
& \chi \in C^2([0,\infty)),\quad \chi >0 \quad \mbox{on} \ [0,\infty), \\ \label{condi;f}
& f\in C^2([0,\infty)), \quad f(0) =0, \quad f>0 \quad \mbox{on} \ (0,\infty),
\end{align}
and moreover, 
\begin{align}\label{condi;fchi}
 \kl{\frac f\chi}' >0, \quad \kl{\frac f\chi}'' \le 0, \quad (\chi f)' \ge 0 \quad \mbox{on} \ [0,\infty)
\end{align}
hold; 
$n_0, c_0, u_0, \Phi, g$ 
are known functions satisfying
 \begin{align}\label{condi;ini1}
   &0 < n_0 
   \in X:= \begin{cases}
    L^{m-1}(\Omega) & \mbox{if} \ m> 2,   
   \\[1mm]   
   L\log L(\Omega) & \mbox{if} \ m \le 2,
   \end{cases} 
 \\[1mm] \label{condi;ini1.5}
   &0 \le c_0 \in L^\infty(\Omega) \ \mbox{such that} \ \sqrt{c_0}\in W^{1,2}(\Omega),   
 \quad 
   u_0 \in L^2_{\sigma}(\Omega), \\ \label{condi;ini2}
   &\Phi \in C^{1+\beta}(\overline{\Omega}), \quad 
   g\in L^2_{\rm loc}([0,\infty); L^{\frac 65}(\Omega)) 
 \end{align}
for some 
$\beta > 0$, where $L^2_\sigma(\Omega):=\{\varphi\in L^2(\Omega) \mid \nabla \cdot \varphi =0\}$.

\abs  
The main result reads as follows. 
%
\begin{thm}\label{mainthm1}
  Let $\Omega\subset\mathbb{R}^3$ be a bounded smooth domain and 
  let $\kappa\in \mathbb{R}$, $\mu\ge 0$, $\alpha >1$. 
Assume that $D$ satisfies \eqref{condi;D} with some $\gamma>0$, $D_2\ge D_1>0$ and $m > 0$,  and that $\chi,f$ satisfy \eqref{condi;chi}{\rm --}\eqref{condi;fchi} as well as that 
  $n_0, c_0, u_0, \Phi, g$ satisfy \eqref{condi;ini1}{\rm --}\eqref{condi;ini2} 
  with some $\beta \in (0,1)$.
  Then, if 
  \begin{align}\label{condi;alpha,m}
     m>\frac 23,\ \mu \ge 0,\ \alpha>1, \quad \mbox{or} \quad 
     m >0,\ \mu >0, \  \alpha > \frac 43
  \end{align}
  hold, 
  there exists a global weak solution $(n,c,u)$ of \eqref{P} in the sense of Definition \ref{def;weaksol}, which can be approximated by a sequence of solutions $(n_{\ep} ,c_{\ep}, u_{\ep})$ of an approximate problem  $($see Section \ref{sec2}\/$)$ 
  in a pointwise manner. 
\end{thm}

\begin{remark}
This theorem gives existence of global weak solutions to \eqref{P}. 
Here we note that, in the previous work \cite{zhang_li}, 
because of lacking regularities of $\na u$, 
constructing the identity $\na \left( \int_0^n D(\sigma)\, d\sigma \right) = D(n)\na n$ seems not to be correct. 
This result shows this identity in the case that $1\le m \le 2$.  
\end{remark}

As an application of this result, we can construct existence result of global weak solutions to \eqref{P} with $\kappa = \mu = 0$, 
which is a correction of the result by Zhang--Li \cite{zhang_li}. 

\begin{corollary}
Let $\Omega\subset\mathbb{R}^3$ be a bounded smooth domain and 
let $\kappa=\mu=0$. 
Assume that $D$ satisfies \eqref{condi;D} with some $\gamma>0$, $D_2\ge D_1>0$ and $m > 0$, and that $\chi,f$ satisfy \eqref{condi;chi}{\rm --}\eqref{condi;fchi} as well as that 
  $n_0, c_0, u_0, \Phi, g$ satisfy \eqref{condi;ini1}{\rm --}\eqref{condi;ini2} with some $\beta \in (0,1)$.
Then, if $m>\frac 23$ holds, 
there exists a global weak solution $(n,c,u)$ of \eqref{P} in the sense of Definition \ref{def;weaksol}. 
Moreover, if $1\le m \le 2$,  then 
$\na \left( \int_0^n D(\sigma)\, d\sigma \right) = D(n)\na n$ holds. 
\end{corollary}

\begin{remark}
 In this result we could verify existence of at least one global weak solution under the condition that $m>\frac 23$. Here we could not include the case that $m=\frac 23$ because of several reasons 
 (see Remarks \ref{remark;zhang-li;young} and \ref{remark;zhang-li;esti}). 
\end{remark}

%

\abs 
The strategy of the proof of Theorem \ref{mainthm1} is to consider approximate problem (see \eqref{Pe}) and to show convergences 
via using arguments similar to those in \cite{Winkler_2016} and \cite{zhang_li}. 
In Section \ref{sec2} we introduce an approximate problem and show several  useful properties for an approximate solution $(\nep,\cep,\uep)$ by using an energy function 
$
 \io \Ne \log \Ne + \frac{1}{2} \io |\na \Psi (\Ce)|^2 + K \io |\Ue|^2
$ 
with some function $\Psi$ and some constant $K>0$, which is used in \cite{Winkler_2016} and \cite{zhang_li}; 
one of keys for a treatment of an energy function is to derive some estimate for $\lp{\frac 65}{\nep}^2$ which comes from a derivative of $\io |\uep|^2$ (see Lemma \ref{lem;energy;u}); in the case that $m> \frac 23$, by virtue of the Gagliardo--Nirenberg inequality, we correct arguments in \cite{zhang_li} and show that 
\[
 \lp{\frac 65}{\nep}^2 \le 
 \eta\lp{2}{\na (\nep +\e)^{\frac m2}}^2 
 + C(\eta) 
\] 
holds with some $C(\eta)>0$ for all $\eta>0$ (see Lemma \ref{lem;problem;dif-C1}); 
on the other hand, in the case that $\mu > 0$ and $\alpha > \frac 43$, from an interpolation argument we have the new estimate:   
\[
  \lp{\frac 65}{\nep}^2 \le 
 C \kl{\mu \io \nep^\alpha+1}   
\]
with some $C>0$ (see Lemma \ref{lem;problem;dif-C2}); then we can establish some differential inequality of an energy function, and obtain several important estimates.  
In Section \ref{sec3} we verify global existence in the approximate problem. Then, aided by estimates obtained in Section \ref{sec2}, 
we can see uniform-in-parameter estimates in Section \ref{sec4}.
Finally, in Section \ref{sec5}, we obtain convergences and establish existence of global weak solutions in \eqref{P}.


\section{An energy time inequality}\label{sec2}

We start by considering the following approximate problem with parameter $\e \in (0,1)$:  
\begin{equation}\label{Pe}
  \begin{cases}
     (n_{\ep})_ t+u_\e\cdot\na n_{\ep} 
     =   \na \cdot (\de (\nep) \na \nep)  
     - \na\cdot\big(\frac{n_{\ep}\chi(\cep)}{1+\e n_{\ep}} \na c_\e\big) + \kappa \nep - \mu \nep^\alpha - \ep \nep^2,
  \\[2mm]
     (c_{\e})_t+u_\e\cdot\na c_\e
     =\D c_\e
     -f(c_\e) \frac{1}{\e}\log\big(1+\e n_{\ep}\big),
  \\[2mm]
     (u_{\e})_t+(Y_\e u_\e\cdot\na)u_\e
     =\D u_\e
     +\na P_\e
     + n_{\ep}\na\Phi +  g_\ep,
     \quad\na\cdot u_\e=0,
  \\[2mm]
     \partial_\nu n_{\ep}|_{\partial\om}
     =\partial_\nu c_\e|_{\partial\om}=0,
     \quad u_\e|_{\partial\om}=0,
  \\[2mm]
     n_{\ep}(\cdot,0)=n_{0\e},\quad
     c_\e(\cdot,0)=c_{0\e},\quad
     u_\e(\cdot,0)=u_{0\e},
  \end{cases}
\end{equation}
where 
\begin{align*}
 \de (s) := D(s+\e) \quad \mbox{for all} \ s\ge 0, 
\quad
  Y_\e=(1+\e A)^{-1}
\end{align*}
and $n_{0\e},c_{0\e}, u_{0\e}, g_\e$ are functions satisfying 
\begin{align}\label{ini;app;nep}
 &n_{0\e} \in C^\infty_0(\Omega),\quad \io n_{0\e} = \io n_0, 
\quad 
 n_{0\e} \to n_0 \ \mbox{in} \ X \ 
 \mbox{as} \ \e \searrow 0,  
\\\label{ini;app;cep}
 &c_{0\e} \in C^\infty_0(\Omega), \quad \lp{\infty}{c_{0\e}}\le \lp{\infty}{c_0}, 
 \quad 
 \sqrt{c_{0\e}} \to \sqrt{c_0} \ \mbox{in} \ L^2(\Omega) \ \mbox{as} \
 \e \searrow 0, 
\\ \label{ini;app;uep}
 &u_{0\e} \in C^\infty_{0,\sigma}(\Omega),
\quad 
 \lp{2}{u_{0\e}}= \lp{2}{u_0}, 
\quad 
 u_{0\e} \to u_0 \ \mbox{in} \ L^2(\Omega) \ \mbox{as} \ \e \searrow 0, 
\\ \notag
 &g_\e \in C^\infty_0(\Omega), \quad 
 \|g_\e\|_{L^2(0,T;L^{\frac 65}(\Omega))} \le 
 \|g\|_{L^2(0,T;L^{\frac 65}(\Omega))} \ \mbox{for all}\ T>0, 
 \\  \label{ini;app;gep}
 &\qquad \qquad \qquad \quad  
 g_\e \to g \ \mbox{in} \ L^2_{\rm loc}([0,\infty); L^{\frac 65}(\Omega)) 
 \ \mbox{as} \ \ep \searrow 0, 
\end{align} 
where 
$A$ is the realization of the Stokes operator in $L^2_\sigma(\Omega)$, 
$X$ is the space 
defined in \eqref{condi;ini1} 
and $C^\infty_{0,\sigma}(\Omega):=\{\varphi\in C^\infty_0(\Omega) \mid 
\nabla \cdot \varphi = 0\}$.  
The first step for the proof of Theorem \ref{mainthm1} is to show global existence of solutions to the approximate problem \eqref{Pe}. 
Now we recall the following result concerned with  
local existence in \eqref{Pe}. 

\smallskip

%
\begin{lem}\label{localsol}
Let $D$ satisfy \eqref{condi;D} with some $\gamma > 0$, $D_2\ge D_1>0$ and $m >0$.  
Assume that $\kappa\in \mathbb{R}$, $\mu\ge 0$, $\alpha >1$, $f\in C^1([0,\infty))$, $\chi\in C^2([0,\infty))$, $\Phi\in C^{1+\beta}(\overline{\om})$ for some $\beta \in (0,1)$ and that $n_{0\ep},c_{0\ep},u_{0\ep},g_\ep$ satisfy \eqref{ini;app;nep}--\eqref{ini;app;gep}.  
  Then for each $\e > 0$ there exist $\etmax \in (0,\infty]$ and 
  uniquely determined functions\/{\rm :}
    \begin{align*}
      n_{\ep}
      &\in C^0(\overline{\om}\times[0,\etmax))
       \cap C^{2,1}(\overline{\om}\times(0,\etmax)),
    \\
       c_\e
      &\in C^0(\overline{\om}\times[0,\etmax))
       \cap C^{2,1}(\overline{\om}\times(0,\etmax))
       \cap L^\infty_{\rm loc}([0,\etmax);W^{1,\infty}(\om)),
    \\
       u_\e
      &\in C^0(\overline{\om}\times[0,\etmax)) 
       \cap C^{2,1}(\overline{\om}\times(0,\etmax)),
    \end{align*}
  which together with some 
  $P_\e\in C^{1,0}(\overline{\om}\times(0,\etmax))$ 
  solve \eqref{Pe} classically. 
  Moreover, $n_{\ep}$ and $c_\e$ are positive 
  and the following alternative holds\/{\rm :} 
  $\etmax=\infty$ or
  \begin{align*}
     \norm{n_{\e}(\cdot,t)}_{L^{\infty}(\om)}
    + \norm{c_\e(\cdot,t)}_{W^{1,q}(\om)}
    + \norm{A^\theta u_\e(\cdot,t)}_{L^2(\om)}
    \to  \infty
  \end{align*}
as $t\nearrow \etmax$ for all $q>3$ and all $\theta\in (\frac 34,1)$. 
\end{lem}
\begin{proof}
Combination of arguments in \cite[Lemma 2.1]{Tao-Winkler_2011_non} 
and \cite[Lemma 2.1]{W-2012}, which is based on 
a standard fixed point argument with a parabolic regularity theory, 
entails this lemma. 
\end{proof}
%

In the following for all $\e\in (0,1)$ 
we denote by $(\nep,\cep,\uep)$ the corresponding solution of \eqref{Pe} given by 
Lemma \ref{localsol} and by $\tmax$ its maximal existence time. 
Then we shall see that $\tmax =\infty$ for all $\ep\in (0,1)$ 
and useful estimates for the approximate solution.  
We first provide the following lemma which is obtained from the first and second equations in \eqref{Pe}. 

\begin{lem}\label{lem;basic}
For all $\ep\in(0,1)$, 
\begin{align*}
\io \nep\cd \le e^{\kappa t}\io n_0 \quad \mbox{for all} \ t\in (0,\tmax)
\end{align*}
and
\begin{align*}
 \mu \int_0^t\io \nep^\alpha + \ep \int_0^t \io \nep^2 \le e^{\kappa t}\io n_0 + \io n_0 \quad \mbox{for all} \ t \in (0,\tmax)   
\end{align*}
as well as 
\begin{align*}
 \lp{\infty}{\cep\cd} \le \lp{\infty}{c_0} 
\quad \mbox{for all} \ t\in (0,\tmax)
\end{align*}
hold. 
\end{lem}
\begin{proof} 
Integrating the first equation in \eqref{Pe}, we have 
 \begin{align*}
 \frac d{dt} \io \nep \le \kappa  \io \nep 
 \end{align*}
on $(0,\tmax)$, which together with \eqref{ini;app;nep} means 
the $L^1$-estimate for $\nep$. 
On the other hand, we apply the maximal principle to the second equation in \eqref{Pe} to obtain that 
\[
 \lp{\infty}{\cep\cd} \le \lp{\infty}{c_{0\e}} \quad \mbox{for all} \ 
 t\in (0,\tmax), 
\]
which with \eqref{ini;app;cep} implies this lemma. 
\end{proof}
\begin{remark}
In order to deal with the case that $\kappa>0$ and $\mu=0$, 
estimates for $\nep$ in this lemma are 
local-in-time estimates which are not often used in the study of the chemotaxis system. 
In the case that $\kappa=\mu=0$ or $\mu>0$, from the well-known arguments 
we can establish a uniform-in-time estimate for $\io \nep$. 
\end{remark}

We then establish estimates for the approximate solution, which are useful not only to see $\tmax = \infty$ for each $\ep\in (0,1)$ but also to obtain uniform-in-$\ep$ estimates, by using an energy function 
defined as 
\[
\io \Ne \log \Ne + \frac{1}{2} \io |\na \Psi (\Ce)|^2 + K \io |\Ue|^2
\] 
with some function $\Psi$ and some constant $K>0$, which is the function same as that used in the previous works \cite{Winkler_2016} and \cite{zhang_li}. 
We first give some estimate for derivatives of the first and second summands in the energy function.  

\begin{lem}\label{lem;energy;nc}
There exists $K>0$ such that for any $\e\in (0,1)$, 
\begin{align*}
 \frac d{dt} \kl{\io \nep \log \nep +\frac 12 \io |\na \Psi (\cep)|^2}
 + \frac 1K \kl{\io \frac {\de(\nep)}{\nep}|\na \nep|^2 + 
 \io \frac{|D^2 \cep|^2}{\cep} + \io \frac{|\na \cep|^4}{\cep^3}}
 \\
  + \io \left( \frac \mu 2 \nep^\alpha + \ep \nep^2 - \kappa \delta_{\mu,0}  \nep \right) 
  \log \nep
 \le K \kl{\io |\na \uep|^2 + 1}
\end{align*}
holds on $(0,\tmax)$, where $\Psi(s):= \int_1^s \frac{d\sigma}{\sqrt{h(\sigma)}}$ with $h(s):= \frac{f(s)}{\chi(s)}$, and 
$\delta_{\mu,0} =1$ when $\mu =0$ and $\delta_{\mu,0} =0$ when $\mu>0$.  
\end{lem}
\begin{proof}
The proof of this lemma 
is similar to those of \cite[Lemma 3.1]{zhang_li} and \cite[Lemmas 2.6 and 2.8]{Lankeit_2016}. 
Aided by arguments in the proof of \cite[Lemma 3.1]{Winkler_2016} and 
noting that 
$(\kappa s -\frac \mu 2 s^\alpha)\log s  \le \kappa\deltamu s\log s + C$ for all $s>0$ with some $C>0$, where $\deltamu$ is the constant defined in the statement of this lemma, 
from straightforward calculations of 
 $\frac d{dt}\io \nep \log \nep$ and $\frac d{dt} \io |\na \Psi (\cep)|^2$ we can verify this lemma. 
\end{proof}

We next calculate a derivative of the third summand $\io |\uep|^2$ in the energy function. 

\begin{lem}\label{lem;energy;u}
There is a constant $C>0$ such that for all $\e\in (0,1)$,  
\begin{align}\label{ineq;badgy}
  \frac 12 \frac d{dt} \io |\uep|^2 + \io |\na \uep|^2 
 &\le C \kl{\lp{\frac 65}{\nep}^2 + \lp{\frac 65}{g_\e}^2 } 
 + \frac 14 \lp{2}{\na \uep}^2
\end{align}
holds on $(0,\tmax)$. 
\end{lem}
\begin{proof}
Testing the third equation of \eqref{Pe} by $\uep$, 
we obtain from the H\"older inequality, the continuous embedding 
$W^{1,2}(\Omega)\hookrightarrow L^6(\Omega)$ and the Young inequality  that 
 \begin{align*}
  \frac 12 \frac d{dt} \io |\uep|^2 + \io |\na \uep|^2 
  & 
  = \io (\nep \na \Phi + g_\e)\cdot \uep   
\\ 
  &\le 
  \kl{\lp{\infty}{\na \Phi}\lp{\frac 65}{\nep} + \lp{\frac 65}{g_\e} }\lp{6}{\uep} 
\\
  & \le 
  C_1 \kl{\lp{\frac 65}{\nep} + \lp{\frac 65}{g_\e} }\lp{2}{\na \uep}    
\\ 
 &\le C_2 \kl{\lp{\frac 65}{\nep}^2 + \lp{\frac 65}{g_\e}^2 } 
 + \frac 14 \lp{2}{\na \uep}^2
\end{align*}  
holds on $(0,\tmax)$ with some $C_1,C_2>0$ which are independent of $\varepsilon$. 
\end{proof} 

In order to derive some differential inequality for the energy function we have to deal with $\lp{\frac 65}{\nep}^2$ in \eqref{ineq;badgy}. 
Now we divide arguments into the cases that 
$m> \frac 23$, $\mu\ge 0$, $\alpha >1$ hold, and that 
$m>0$, $\mu>0$, $\alpha > \frac 43$ hold. 
We first deal with the case that $m>\frac 23$, $\mu \ge 0$, $\alpha >1$ hold. 
In this case we use the diffusion effect to control $\lp{\frac 65}{\nep}$. 
The proof of the following lemma is based on that of \cite[Lemma 3.2]{zhang_li}. 
However, in order to see the following lemma, 
we need the restriction of $m> \frac 23$ 
instead of $m \ge \frac 23$ which is assumed in \cite{zhang_li} (see Remark \ref{remark;zhang-li;young}). 

\begin{lem}\label{lem;problem;dif-C1}
Assume that $m > \frac 23$, $\mu \ge 0$, $\alpha >1$. 
Then for all $T >0$ and all $\eta >0$ 
there is $C(T,\eta)>0$ such that 
\[
 \lp{\frac 65}{\nep }^2 \le 
 \eta\lp{2}{\na (\nep +\e)^{\frac m2}}^2 
 + C(T,\eta)
\]
for all $t\in (0,\tilt)$ and all $\ep \in (0,1)$, where $\tilt:= \min\{T,\tmax\}$. 
\end{lem}
\begin{proof}
The proof is similar to that of \cite[Lemma 3.2]{zhang_li}. 
Let $T >0$ and put $\tilt := \{T,\tmax\}$. 
Noting from Lemma \ref{lem;basic} that 
\[
 \lp{\frac 2m}{(\nep+\e)^{\frac m2}}= \lp{1}{\nep+\e}^\frac{m}{2}
 \le C_1(T)
\]
for all $t\in (0,\tilt)$ and all $\ep\in (0,1)$ with some $C_1(T)>0$, 
we use the Gagliardo--Nirenberg inequality 
to see that 
\begin{align*}
 \lp{\frac 65}{\nep}^2 
 &\le \lp{\frac 65}{\nep +\e }^2 
 \\
 &= \lp{\frac {12}{5m}}{(\nep + \e)^{\frac m2}}^{\frac 4m} 
\\
 &\le C_{2} \lp{2}{\na (\nep +\e)^{\frac m2}}^{\frac 2{3m-1}}
 \lp{\frac 2m}{(\nep+\e)^{\frac m2}}^{\frac {5m-2}{3m-1}} 
 + \lp{\frac 2m}{(\nep+\e)^{\frac m2}}^{\frac 4m}
\\
 &\le 
 C_3(T) \kl{\lp{2}{\na (\nep +\e)^{\frac m2}}^{\frac 2{3m-1}}+ 1}
\end{align*}
for all $t\in (0,\tilt)$ and all $\e\in (0,1)$ with some $C_{2}>0$ and some $C_3(T)>0$. 
Now, since the condition $m>\frac 23$ implies that $\frac 2{3m-1}< 2$, 
we establish from the Young inequality that 
\begin{align*}
 \lp{\frac 65}{\nep}^2 \le 
 \eta \lp{2}{\na (\nep +\e)^{\frac m2}}^2 
 + C_4(T,\eta)
\end{align*} 
for all $t\in (0,\tilt)$ and all $\e\in(0,1)$ with some $C_4(T,\eta)>0$ for all $\eta>0$. 
\end{proof}
\begin{remark}\label{remark;zhang-li;young}
In  \cite[(3.6)]{zhang_li}, 
they used the Young inequality as 
\[
\kl{\io (\nep + \ep)^{m-2}|\na \nep|^2}^{\frac 1{3m-1}} 
\le \eta \io  (\nep + \ep)^{m-2}|\na \nep|^2 + C(\eta,m)
\]
with some $C(\eta,m)>0$ for all $\eta>0$ and for $m\ge \frac 23$ 
even though this inequality does not hold when $m=\frac 23$. Indeed, since $\frac 1{3m-1} = 1$ holds when $m=\frac 23$, we could not apply the Young inequality to $\io (\nep + \ep)^{m-2}|\na \nep|^2$. 
Thus we need to assume that $m>\frac  23$ when we use this method.  
\end{remark}

We then consider the case that $m>0$, $\mu >0$, $\alpha >\frac 43$ hold. 
In this case we use the logistic-type damping to control $\lp{\frac 65}{\nep}$. 

\begin{lem}\label{lem;problem;dif-C2}
Assume that $m >0$, $\mu >0$, $\alpha> \frac 43$. 
Then for all 
$T > 0$  
there is $C(T)>0$ such that 
\[
 \lp{\frac 65}{\nep}^2 \le 
 C(T)\kl{\mu \io \nep^\alpha+1}   
\]
holds for all $t\in (0,\tilt)$ and all $\ep\in (0,1)$, 
where $\tilt := \min\{T,\tmax\}$. 
\end{lem}
\begin{proof}
Let $T >0$ and put $\tilt :=\min\{T,\tmax\}$. 
We use an interpolation inequality and the Young inequality to obtain that 
\begin{align*}
\lp{\frac 65}{\nep}^2 \le 
\lp{\alpha}{\nep}^{\frac{2\alpha}{6(\alpha-1)}}\lp{1}{\nep}^{\frac{5\alpha-6}{6(\alpha-1)}}
\le C_1(T)\kl{\mu \lp{\alpha}{\nep}^\alpha+1}
\end{align*}
on $(0,\tilt)$ 
with some $C_1(T)>0$, 
which with Lemma \ref{lem;basic} and the fact $\frac{2\alpha}{6(\alpha-1)} < \alpha$ (from $\alpha > \frac 43$) implies this lemma. 
\end{proof}

The lemmas obtained in this section yield the following estimate for a derivative of the energy function. 

\begin{lem}\label{lem;energy}
Let $\Psi$ and $K>0$ be given in Lemma \ref{lem;energy;nc} 
and assume that \eqref{condi;alpha,m} holds. 
Then for all $T >0$ there is $C(T)>0$ such that for any $\e\in (0,1)$,  
\begin{align*}
 &\frac{d}{dt} \kl{\io \nep \log \nep +\frac 12 \io |\na \Psi (\cep)|^2 
 + K \io |\uep|^2 }
 + \io \left( \frac \mu 2 \nep^\alpha + \ep \nep^2 \right)\log \nep
 \\ 
 &\quad \, 
 + \frac 1{2K} \kl{
 \io \frac{\de(\nep)}{\nep}  |\na \nep|^2+ 
 \io \frac{|D^2 \cep|^2}{\cep} + \io \frac{|\na \cep|^4}{\cep^3}+ 
 \io |\na \uep|^2}
 \\
 &\le C(T)\kl{1+\mu \io \nep^\alpha + \lp{\frac 65}{g_\e\cd }^2} + \kappa \deltamu \io \nep\log \nep  
\end{align*}
holds on $(0,\tilt)$ with $\tilt :=\min\{ T,\tmax\}$. 
\end{lem}
\begin{proof}
Let $T>0$ and put $\tilt :=\min\{ T ,\tmax\}$. 
Aided by Lemmas \ref{lem;energy;u}, \ref{lem;problem;dif-C1} and \ref{lem;problem;dif-C2}, 
we can obtain that 
 \begin{align}\label{ineq;dif;u}\notag
  \frac 12 \frac d{dt} \io |\uep|^2 + \frac 34 \io |\na \uep|^2 
  &\le \frac{D_1}{4K^2}\io (\nep +\e)^{m-2}|\na \nep|^2 
  \\
 &\quad\, + C_1(T)\kl{\lp{\frac 65}{g_\e}^2+\mu \io \nep^\alpha +1}
\end{align}  
for all $t\in (0,\tilt)$ and 
for all $\e\in(0,1)$ with some $C_1(T) >0$. 
Thus a combination of Lemma \ref{lem;energy;nc} and 
\eqref{ineq;dif;u} derives this lemma. 
\end{proof}

In the end of this section we provide the following uniform-in-$\ep$ estimates for the approximate solution which will be used later. 

\begin{lem}\label{lem;esti1}
Let $\Psi$ be given in Lemma \ref{lem;energy;nc} 
and assume that  \eqref{condi;alpha,m} holds. 
Then for all $T >0$ there exists $C(T)>0$ such that 
\begin{align*}
 \io \nep \log \nep + \io |\na \Psi (\cep)|^2 
 +  \io |\uep|^2 \le C(T) 
\quad \mbox{for all} \ t\in (0,\tilt)
\end{align*}
and 
\begin{align}\label{esti;spacetimeintegral}
 \iiio \frac{\de(\nep)}{\nep} |\na \nep|^2 
 + \iiio \left( \mu \nep^\alpha + \ep \nep^2 \right) \log \nep 
  \le C(T) 
\end{align}
as well as 
\begin{align*}
  \iiio \frac{|D^2 \cep|^2}{\cep} 
 + \iiio \frac{|\na \cep|^4}{\cep^3}
 + \iiio |\na \uep|^2 
  \le C(T) 
\end{align*}
hold for all $\e\in (0,1)$ with $\tilt:= \min\{T,\tmax\}$. 
\end{lem} 
\begin{proof}
Let $T>0$ and put $\tilt := \min\{T,\tmax\}$, 
and let $K$ be given in Lemma \ref{lem;energy;nc}. 
Putting 
\[
 \ye (t) := \io \nep\cd \log \nep\cd +\frac 12 \io |\na \Psi (\cep\cd)|^2 
 + K \io |\uep\cd|^2  
\]
and 
\begin{align*}
 \ze (t) &:= 
  \io \frac{\de(\nep \cd)}{\nep\cd} |\na \nep\cd|^2+ 
 \io \frac{|D^2 \cep\cd|^2}{\cep\cd} + \io \frac{|\na \cep\cd|^4}{\cep^3\cd}+ 
 \io |\na \uep\cd|^2
\\
& \quad\, + 2K \io \kl{\frac \mu 2 \nep^\alpha\cd + \ep \nep^2\cd }\log \nep\cd
\end{align*}
for $t\in (0,\tilt)$,  
we obtain from Lemma \ref{lem;energy} that 
\begin{align}\notag
\ye'(t) &+ \frac 1{2K}\ze(t) -\kappa \deltamu \io \nep\cd \log \nep\cd 
\\ \label{ineq;dif;yz}
& \le C_1(T) \kl{1+ \mu \io \nep^\alpha\cd +\lp{\frac 65}{g_\e \cd}^2}
\quad \mbox{for all} \ t\in (0,\tilt)
\end{align}
with some $C_1(T)>0$.
Then, in order to derive a differential inequality of $\ye$, 
we shall show that 
\[
 C \ye (t) + \kappa\deltamu \io \nep\cd \log \nep \cd \le \frac 1{2K} \ze (t) +\widetilde{C} 
  \quad \mbox{for all} \ t\in (0,\tilt)
\]
with some $C,\widetilde{C}>0$. 
Now, in the case that $m>\frac 23$, the inequality 
\[
 s\log s \le \frac 3{3m-2} s^{m+\frac 13} 
\quad \mbox{for all} \ s>0 
\]
and the Gagliardo--Nirenberg inequality 
entail that 
\begin{align*}
 (\kappa\deltamu &+1)\io \Ne \log \Ne 
\\
 & \le \frac{3(\kappa\deltamu +1)}{3m-2} \io (\Ne+\e)^{m+\frac 13}
\\
 & \le C_{2} \kl{\lp{2}{\na (\Ne +\e)^\frac m2}^{\frac{2(3m-2)}{3m-1}}\lp{\frac 2m}{(\Ne+\e)^\frac m2}^\frac{2(6m-1)}{3m(3m-1)} 
 + \lp{\frac 2m}{(\Ne+\e)^\frac m2}^\frac{2(3m+1)}{3m}}
\\
 &\le \frac{D_1}{2K} \io (\Ne+\e)^{m-2}|\na \Ne|^2 + C_3(T)
 \\
 & 
 \le \frac 1{2K}\io \frac{\de(\nep)}{\nep}|\na \nep|^2+C_3(T)
\end{align*}
for all $\e\in (0,1)$ with some $C_{2}, C_3(T) >0$. 
On the other hand, in the case that $\mu >0$ (which means that $\deltamu =0$ holds), 
the inequality 
\[
  s\log s \le  s^\alpha \log s \quad \mbox{for all} \ s>0
\]
enables us to see that 
\[
 \io \nep \log \nep 
\le 
 \io \nep^\alpha \log \nep. 
\]
Moreover, by putting 
$M:= \min \{ h'(s)\mid s\in [0,\lp{\infty}{c_0}]\} > 0$  
and using the inequality 
\[
 h(s) \ge M s \quad \mbox{for all} \ s \in [0,\lp{\infty}{c_0}]
\]
(from the facts that $h\in C^1([0,\infty))$, $h'>0$ on $[0,\lp{\infty}{c_0}]$ and $h(0)=0$),  
we can see that 
\begin{align*}
 \frac 12 \io |\na \Psi (\Ce)|^2 
 &=  \frac 12 \io \frac{|\na \Ce|^2}{h(\Ce)} 
\\
 &\le \frac 14 \io \frac{|\na \Ce|^4}{\Ce^3} 
 + \frac 14 \io \frac{\Ce^3}{h^2(\Ce)}
\\
 &\le \frac 14 \io \frac{|\na \Ce|^4}{\Ce^3} + 
  \frac{\lp{\infty}{c_0}|\Omega|}{4M^2} 
\end{align*}
holds for all $t\in(0,\tilt)$. 
Therefore the Poincar\'e inequality 
\[
 K \io |\Ue|^2 \le C_4 \io |\na \Ue|^2 
\]
with some $C_4>0$ and the inequality $-\frac 1e \le \ep s\log s \le \ep s^2 \log s$ for all $s>0$ and all $\ep\in (0,1)$ 
yield that 
 \begin{align}\label{ineq;rela;yz}
 C_5(T) \ye (t) + \kappa\deltamu \io \nep\cd \log \nep \cd \le \frac 1{2K} \ze (t) + C_6(T) 
  \quad \mbox{for all} \ t\in (0,\tilt)
\end{align}
with some $C_5(T),C_6(T) >0$. 
Since \eqref{ineq;dif;yz} and \eqref{ineq;rela;yz} derive that 
\begin{align*}
\ye'(t) + C_7(T)\ye(t) \le C_8(T) \kl{1+ \mu \io \nep^\alpha\cd +\lp{\frac 65}{g_\e \cd}^2}
\end{align*}
for all $t\in (0,\tilt)$ with some $C_7(T),C_8(T)>0$,
the existence of $C_{9}(T)>0$ satisfying  
\[
\int_0^{\tilt}\io \nep^\alpha + \int_0^{\tilt} \lp{\frac 65}{g_\e}^2 \le C_{9}(T),
\]
which is obtained from \eqref{condi;ini2} and \eqref{ini;app;gep} as well as Lemma \ref{lem;basic}, 
means that this lemma holds. 
\end{proof}

\section{Global existence for the regularized problem \eqref{Pe}}\label{sec3}

In this section we show global existence in the approximate problem by using the estimates obtained in Lemma \ref{lem;esti1}. 

\begin{lem}\label{lem;ge}
Assume that \eqref{condi;alpha,m} holds. 
Then for all $\e\in (0,1)$, $\tmax = \infty$ holds. 
\end{lem}
\begin{proof}
Arguments similar to those in the proof of 
\cite[Lemma 2.9]{KM} enable us to see this lemma; thus we only write a short proof. 
Assume that $\tmax < \infty$.  
We can obtain from Lemma \ref{lem;esti1} with $T =\tmax$ that 
$\io |\Ue|^2 \le C_1$ 
for all $t\in (0,\tmax)$ 
and 
\begin{align*}
 \int_0^{\tmax}\io|\na \Ce|^4 \le \lp{\infty}{c_0}^3\int_0^{\tmax} \io \frac{|\na \Ce|^4}{\Ce^3} \le C_2  
\end{align*}
with some $C_1,C_2>0$ (which are independent of $\varepsilon$). 
Now we let $p:= \min\{3+m, 4\}>3$. 
Then, considering $\frac d{dt}\io (\Ne+\e)^p$ together with the inequality $\frac{s}{1+\e s} \le \frac 1\e$ for all $s>0$, 
we can verify an $L^p(\Omega\times (0,\tmax))$-estimate for $\nep$.  
Then, through boundedness of $\sup_{t\in (0,\tmax)}\|\uep \cd\|_{D(A^\theta)}$ with some $\theta\in (\frac 34,1)$ and $\sup_{t\in (0,\tmax)}\|\na \cep \cd\|_{L^6(\Omega)}$, 
a Moser--Alikakos-type procedure (see the proof of \cite[Lemma A.1]{Tao-Winkler_2012}) enables us to have an $L^\infty(\Omega\times (0,\infty))$-estimate for $\nep$, which with the extensibility criterion implies  that $\tmax = \infty$ for all $\ep\in (0,1)$. 
\end{proof}

\section{Further $\bm{\e}$-independent estimates for \eqref{Pe}}\label{sec4}

In this section we derive uniform-in-$\ep$ estimates for the approximate solution which will be used in Section \ref{sec5}. 
We first give the following estimates. 

\begin{lem}\label{lem;esti2-1}
Assume that \eqref{condi;alpha,m} holds.  
Then for all $T>0$ there exists $C(T)>0$ such that 
\begin{align}\label{esti;lem2-1}
 \iio |\na (\Ne+\e)^\frac m2  |^2
 + \iio |\na \Ce|^4
 + \iio |\Ue|^\frac{10}{3} 
 \le C(T)
\end{align}
holds for all $\ep\in (0,1)$. 
\end{lem}
\begin{proof}
The proof is based on arguments in the proof of \cite[Lemma 5.1]{zhang_li}. 
Thus we only write a short proof. 
Let $T>0$. 
Due to \eqref{esti;spacetimeintegral}, 
we can find $C_1(T)>0$ such that  
\begin{align*}
 \iio |\na (\Ne+\e)^\frac m2|^2 
 & = \frac{m^2}{4}
 \iio (\Ne+\e)^{m-2}|\na \Ne|^2 
\\
 &\le 
 \frac{m^2}{4D_1}\iio \frac{\de(\nep)}{\nep}|\na \nep|^2 
 \\
 & 
 \le C_1(T)
\end{align*} 
for all $\ep\in (0,1)$. 
On the other hand, 
in light of Lemma \ref{lem;basic} 
and the Gagliardo--Nirenberg inequality, 
we infer from Lemma \ref{lem;esti1} that 
\begin{align*}
 \iio |\na \Ce|^4 & \le \lp{\infty}{c_0}^3 \iio \frac{|\na \Ce|^4}{\Ce^3} 
 \\
 & \le C_2(T) 
\end{align*}
and 
\begin{align*}
 \iio |\Ue|^\frac{10}{3} & \le C_3 \int_0^T \kl{\lp{2}{\na \Ue}^2 \lp{2}{\Ue}^{\frac 43}+ \lp{2}{\Ue}^\frac{10}{3}}
\\
 & \le C_4(T)
\end{align*}
for all $\ep\in (0,1)$ with some $C_2(T),C_3,C_4(T)>0$. 
\end{proof}

In the following, we only consider the case that 
\begin{align}\label{condi;case1}
m > \frac 23, \quad \mu \ge 0, \quad \alpha > 1
\end{align}
hold, or that 
\begin{align}\label{condi;case2}
m\in \left(0,\frac 23\right], \quad \mu >0, \quad \alpha >\frac 43
\end{align}
hold. Here we note that,  
since 
\[
 \left \{ (m,\mu,\alpha) \,\Big{|}\, m>\frac 23, \ \mu > 0, \ \alpha > \frac 43 \right\}
  \subset \left\{ (m,\mu,\alpha) \, \Big{|} \, m> \frac 23, \ \mu \ge 0, \ \alpha >1 \right\} 
\]
holds, it is enough to consider the case that  
\eqref{condi;case1} or \eqref{condi;case2} holds when \eqref{condi;alpha,m} holds.   

\subsection{Key estimates. Case 1: $m > \frac 23$, $\mu \ge 0$, $\alpha>1$}

In this subsection we establish estimates for $\nep$ in 
the case that \eqref{condi;case1} holds. 
In this case by using the diffusion effect 
we can obtain the following estimates. 

\begin{lem}\label{lem;esti2-2}
Assume that \eqref{condi;case1} holds. 
Then for all $T>0$ there exists $C(T)>0$ such that 
\begin{align}\label{esti;lem2-2} 
 \iio (\Ne + \e)^{\frac{3m+2}{3}}
 + \iio |\de(\Ne)\na \Ne|^{\frac{3m+2}{3m+1}}
 \le C(T),
\end{align}
and moreover, if $\frac 23 < m \le 2$, then 
\begin{align*}
 \iio |\na \Ne|^\frac{3m+2}{4} \le C(T)
\end{align*} 
hold for all $\ep\in (0,1)$. 
\end{lem}
\begin{proof}
The proof is similar to that of \cite[Lemma 5.1]{zhang_li}. 
Thus again we only write a short proof. 
Let $T>0$.  
The estimate \eqref{lem;esti2-1} yields from 
the Gagliardo--Nirenberg inequality 
that 
\begin{align*}
 \iio (\Ne+\e)^\frac{3m+2}{3} 
 &\le C_{1} \int_0^T\! 
 \kl{\lp{2}{\na (\Ne+\e)^\frac m2}^2\lp{\frac 2m}
 {(\Ne+\e)^\frac m2}^ \frac 4{3m} 
 + \lp{\frac 2m}{(\Ne+\e)^\frac m2}^\frac{2(3m+2)}{3m}}	
\\
 &\le C_2(T)
\end{align*}
for all $\ep\in (0,1)$ with some $C_{1}>0$ and $C_2(T)>0$. 
Then, we use the H\"older inequality and 
\eqref{esti;spacetimeintegral} to confirm that 
\begin{align*}
 \iio |\de(\Ne)\na \Ne|^{\frac{3m+2}{3m+1}} 
 &\le C_3\kl{\iio \frac{\de(\Ne)}{\Ne}|\na \Ne|^2}^\frac{3m+2}{6m+2}
 \kl{\iio (\Ne+\e)^\frac{3m+2}{3}}^{\frac{3m}{6m+2}} 
\\
 &\le C_4(T)  
\end{align*}
for all $\ep\in (0,1)$ 
with $C_3:= D_2^{\frac{3m+2}{6m+2}}>0$ and some $C_4(T)>0$, which implies \eqref{esti;lem2-2} holds. 
Moreover, if $\frac 23 < m\le 2$, 
then a combination of 
the Young inequality and \eqref{esti;lem2-1}, along with 
\eqref{esti;lem2-2} leads to existence of the constant 
$C_5(T)>0$ such that for all $\ep\in (0,1)$, 
\begin{align*}
 \iio |\na \Ne|^\frac{3m+2}{4} 
 \le \iio (\Ne+\e)^{m-2}|\na \Ne|^2 + \iio (\Ne+\e)^\frac{3m+2}{3} 
 \le C_5(T), 
\end{align*}
which completes the proof. 
\end{proof}
\begin{remark}
In this lemma we establish the boundedness of $\iio |\na (\Ne + \ep)^\frac{m}{2}|^2$ instead of $\iio |\na \Ne^\frac{m}{2}|^2$ which was shown in \cite[Lemma 5.1]{zhang_li}, because the inequality 
\[
\iio \Ne^{m-2}|\na \Ne|^2 \le \iio (\Ne +\ep)^{m-2}|\na \Ne|^2,
\] 
which was utilized to show boundedness of $\iio |\na \Ne^\frac m2|^2$ 
in the proof of \cite[Lemma 5.1]{zhang_li}, 
seems to hold only when $m\ge 2$.  
\end{remark}

\subsection{Key estimates. Case 2: $m\in (0,\frac 23]$, $\mu>0$, $\alpha >\frac 43$}

We next deal with the case that \eqref{condi;case2} holds. 
In this case we can obtain important estimates for $\nep$ from the logistic-type damping. 

\begin{lem}\label{lem;esti2-3}
Assume that \eqref{condi;case2} holds.  
Then for all $T>0$ there exists $C(T)>0$ such that 
\begin{align*}
 \iio (\Ne + \e)^{\alpha}
 + \iio |\de(\Ne)\na \Ne|^{\frac{2\alpha}{\alpha + m}}
 + \iio |\na \nep|^{\frac{2\alpha}{2+\alpha -m}}
 \le C(T)
\end{align*}
holds for all $\ep\in (0,1)$. 
\end{lem}
\begin{proof}
Let $T>0$. Noticing that $\alpha >\frac 43 (>1)$ and $\ep\in (0,1)$, 
from Lemma \ref{lem;basic} we can find $C_1(T)>0$ such that 
\begin{align}\label{ineq;LLalpha}
\iio (\nep+\e)^\alpha \le 2^{\alpha-1} \iio (\nep^\alpha+1) \le  C_1(T) 
\end{align}
for all $\ep\in (0,1)$. 
Then the H\"older inequality, 
\eqref{condi;D} and \eqref{esti;spacetimeintegral} enable us to have that 
\begin{align*}
\iio |\de(\Ne)\na \Ne|^{\frac{2\alpha}{\alpha + m}} 
&\le \kl{\iio \frac{\de(\Ne)}{\Ne}|\na \Ne|^2}^{\frac{\alpha}{m+\alpha}} 
 \kl{\iio \de^{\frac{\alpha}{m}}(\nep) \cdot\nep^{\frac \alpha m}}^{\frac{m}{\alpha +m}} 
\\
 &\le C_2\kl{\iio \frac{\de(\Ne)}{\Ne}|\na \Ne|^2}^{\frac{\alpha}{m+\alpha}} 
 \kl{\iio (\Ne+\e)^\alpha}^{\frac{m}{\alpha +m}} 
\\
 &\le C_3(T)  
\end{align*}
for all $\ep\in (0,1)$ with $C_2:= D_2^{\frac{\alpha}{\alpha + m}}>0$ and some $C_3(T)>0$. 
Moreover, we establish from the Young inequality that 
\begin{align*}
\iio |\na \nep|^{\frac{2\alpha}{2+\alpha-m}} \le 
\iio (\nep+\e)^{m-2}|\na \nep|^2 + \iio (\nep+\ep)^\alpha,
\end{align*} 
which with Lemma \ref{lem;esti2-1} and \eqref{ineq;LLalpha} concludes 
the proof. 
\end{proof}

\subsection{An estimate for $\nep \uep$}

In summary, in both cases that \eqref{condi;case1} holds and that \eqref{condi;case2} holds, 
we verified the following important estimates for $\nep$, 
which are cornerstones in the proof of Theorem \ref{mainthm1}.  

\begin{lem}\label{lem;esti2-4} 
Assume that \eqref{condi;case1} or \eqref{condi;case2} holds. 
Then for all $T>0$ there is $C(T)>0$ such that 
\begin{align}\label{esti;lem2} 
 \iio (\Ne + \e)^{p_1}
 + \iio |\de(\Ne)\na \Ne|^{p_2}
 \le C(T), 
\end{align} 
where $p_1 >\frac 43$ and $p_2 \in (1,2)$ are constants defined as 
\begin{align}\label{def;p1p2}
p_1:= \begin{cases}
\frac{3m+2}{3} & \mbox{if \eqref{condi;case1} holds}, 
\\ 
\alpha & \mbox{if \eqref{condi;case2} holds} 
\end{cases}
\quad \mbox{and} \quad 
p_2:= \begin{cases}
\frac{3m+2}{3m+1} & \mbox{if \eqref{condi;case1} holds}, 
\\ 
\frac{2\alpha}{\alpha + m} & \mbox{if \eqref{condi;case2} holds}. 
\end{cases}
\end{align}
Moreover, if $0 < m \le 2$, for all $T>0$ 
\[
\iio |\na \nep|^{p_3} \le \widetilde{C}(T) 
\]
holds for all $\ep\in (0,1)$ with some $\widetilde{C}(T)>0$, where 
\begin{align*}
 p_3:= \begin{cases}
  \frac{3m+2}{4} & \mbox{if \eqref{condi;case1} holds}, 
  \\ 
  \frac{2\alpha}{2+ \alpha - m} & \mbox{if \eqref{condi;case2} holds.}
 \end{cases}
\end{align*}
\end{lem}
\begin{proof}
By virtue of Lemmas \ref{lem;esti2-2} and \ref{lem;esti2-3}, we can obtain the estimates in the statement. 
\end{proof}

\begin{remark}\label{remark;zhang-li;esti}
In order to obtain estimates for $\nep$ stated in Lemma \ref{lem;esti2-4} we need to assume that $m>\frac 23$ or that $\alpha >\frac 43$ (with $\mu>0$). Indeed, 
if we assume that $m>\frac 23$ or that $\alpha>\frac 43$, then we can confirm that $p_1 > \frac 43$ holds, which is important when we consider convergences of the approximate solution (see Lemma \ref{lem;conv2} and its proof). 
\end{remark}
In the proof of Theorem \ref{mainthm1} we also need to establish some estimate for $\Ne \Ue$ (see e.g., Proof of Lemma \ref{lem;esti5}). 
Here, in the case that $m>\frac 23$, the previous work \cite{zhang_li} asserts from Lemmas  \ref{lem;esti2-1} and \ref{lem;esti2-2} that 
\[ 
 \iio |\Ne\Ue| \le \iio \Ne^{\frac{10}{7}} + \iio |\Ue|^{\frac {10}{3}}
 \le C(T)
\]
with some $C(T)>0$; 
however, $\iio \Ne^{\frac{10}{7}}$ is bounded only when 
$(\frac 23 < )\frac{16}{21}\le m $ (from the relation $\frac{10}{7} \le \frac{3m+2}{3}$). 
Therefore we need the following additional estimate to control the term $\iio \Ne \Ue$. 

\begin{lem}\label{lem;esti3}
Assume that \eqref{condi;case1} or \eqref{condi;case2} holds. 
Then for all $T>0$ there exists $C(T)>0$ such that 
\begin{align*}
\int_0^T \lp{r}{\Ne}^2 + \int_0^T \lp{\frac{rq}{r-q}}{\Ue}^{\frac{2q}{2-q}} \le C(T) 
\quad \mbox{for all} \ \ep \in (0,1) 
\end{align*}
with some $r > \frac{6}{5}$ and some $q>1$, 
and moreover, 
\[
\int_0^T \int_\Omega |\Ne\Ue|^q  
\le C
\quad \mbox{for all} \ \ep\in (0,1) 
\]
holds. 
\end{lem}
\begin{proof}
Let $T>0$. 
Now, since $\frac{3\cdot \frac 65 -2}{\frac 65} = \frac 43 < p_1$, we can find $r\in (\frac 65,\min\{p_1,2\})$ such that 
$
\frac{3r-2}{r} \le p_1, 
$ 
which implies that 
\begin{align}\label{ineq;rela;rp_1}
2\cdot \frac{r-1}{r}\cdot \frac{p_1}{p_1-1} \le p_1 
\end{align}
holds. 
Therefore an interpolation inequality with \eqref{ineq;rela;rp_1}  entails from \eqref{esti;lem2} that 
\begin{align}\label{esti;nLr2}\notag
 \int_0^T \lp{r}{\Ne}^2 
 &\le  \int_0^T \lp{p_1}{\nep}^{\frac{2(r-1)p_1}{r(p_1-1)}}\lp{1}{\nep}^{\frac{2(p_1-r)}{r(p_1-1)}}
\\
 &\le C_1(T) \int_0^T \kl{\lp{p_1}{\nep}^{p_1}+1}\le C_2(T)
\end{align}
for all $\ep\in (0,1)$ with some $C_1(T),C_2(T)>0$. 
On the other hand, since the fact $r\in (\frac 65,2)$ implies that 
$0< \frac{3(2-r)}{2r} < 1$ holds, 
we can find $q\in (1,2)$ such that
 \[
 a:= \frac{3(rq-2r+2q)}{2rq} \in (0,1)
 \]
and 
 \[
  \frac{2q}{2-q}\cdot a <2. 
 \]
Then the  Gagliardo--Nirenberg inequality, Lemma \ref{lem;esti1} and the Young inequality derive that 
 \begin{align}\label{esti;LblaLblaforu}\notag
 \int_0^T \lp{\frac{rq}{r-q}}{\Ue}^{\frac{2q}{2-q}} 
 &\le C_3 \int_0^T \kl{\lp{2}{\na \Ue}^{\frac{2q}{2-q}a} \lp{2}{\Ue}^{\frac{2q}{2-q}(1-a)} + \lp{2}{\Ue}^{\frac{2q}{2-q}}} 
 \\ \notag
 &\le C_4(T) \kl{ \int_0^T \lp{2}{\na \Ue}^2+1 }
\\ 
 &\le C_5(T)
 \end{align}
 with some $C_3,C_4(T),C_5(T)>0$, 
where we have used the relation $\frac{2q}{2-q}a<2$. 
Moreover, we infer from 
the H\"{o}lder inequality, 
\eqref{esti;nLr2} and 
\eqref{esti;LblaLblaforu}
that 
\[
\int_0^T \io |\Ne \Ue|^q 
\le \kl{\int_0^T \lp{r}{\Ue}^2}^\frac{q}{2}\kl{\int_0^T \lp{\frac{rq}{r-q}}{\Ue}^{\frac{2q}{2-q}}}^{\frac{2-q}{2}} 
\le C_2(T)^\frac{q}{2}C_5(T)^{\frac{2-q}{2}} 
\]
holds, which completes the proof. 
\end{proof}

\subsection{Time regularities}

One of strategies for establishing convergences of the approximate solution  is to use an Aubin--Lions-type lemma (cf.\ \cite[Corollary 4]{Simon_1987}). To apply an Aubin--Lions-type lemma to $(\nep)_{\ep\in (0,1)}$ we desire some estimates for $\na \nep$ and $\partial_t \nep$; however, in view of Lemma \ref{lem;esti2-4}, we could have some estimate for $\na \nep$ only in the case that $m\in (0,2]$. 
Thus, in the case that $m>2$, we need to use some different quantity e.g., $(\nep+\ep)^\gamma$ with some $\gamma>0$. 
By virtue of Lemma \ref{lem;esti2-1}, $(\nep+\ep)^\frac m2$ is one of candidates of this quantity. 
Now we show the following lemma which is 
utilized to obtain an estimate for $\partial_t (\nep+\ep)^{\frac m2}$ when $m>2$. 

\begin{lem}\label{lem;esti4}
If $m>2$, then for all $T>0$ there exists $C(T)>0$ such that 
\begin{align*}
 \io (\Ne + \ep)^{m-1} \le C(T)  
 \quad \mbox{for all} \ t\in (0,T)\ \mbox{and all} \ \ep \in (0,1)
\end{align*}
and 
\[
 \iio |\na (\Ne+\ep)^{m-1}|^2 \le C(T) 
 \quad \mbox{for all} \ \ep\in (0,T). 
\]
\end{lem}
\begin{proof}
The main strategy for the proof is based on that in the proof of \cite[Lemma 3.3]{KM}. 
Since the inequality $(m-1)(m-2)>0$ holds from the condition $m>2$, 
the condition \eqref{condi;D} and the H\"older inequality yield that 
\begin{align*}
 \frac{d}{dt} \io (\Ne + \ep)^{m-1} 
 & = -(m-1)(m-2) \io (\Ne+\ep)^{m-3}\de (\Ne) |\na \Ne|^2 
 \\
 &\quad\, + (m-1)(m-2) \io \frac{(\Ne+\ep)^{m-3}\Ne}{1+\ep \Ne}\chi(\Ce)\na \Ne\cdot \na \Ce
\\ 
 & \le - \frac{D_1(m-2)}{2(m-1)} \io |\na (\Ne+\ep)^{m-1}|^2 
   + \frac{(m-1)(m-2)}{2D_1} \io |\na \Ce|^2. 
\end{align*}
Then, for each $T>0$ and all $t\in (0,T)$, 
integrating it over $(0,t)$ derives that 
\begin{align*}
 \io (\Ne + \ep)^{m-1} 
 &+ \frac{D_1(m-2)}{2(m-1)} \int_0^t \io |\na (\Ne+\ep)^{m-1}|^2 
\\
 &\le 
 \io (n_{0,\ep} + \ep)^{m-1} 
 + \frac{(m-1)(m-2)}{2D_1} \iio |\na \Ce|^2, 
\end{align*} 
which implies that this lemma holds. 
\end{proof}
\begin{remark}
This lemma is similar to \cite[Lemma 5.2]{zhang_li}; 
however, 
there is a mistake in the class of the initial data $n_0\in L\log L(\Omega)$ when $m>2$; indeed, to estimate $\io n_{0\ep}^{p}$ for all $p\in [1,9(m-1))$, 
we have to assume that $n_0\in \bigcap_{p\in [1,9(m-1))} L^{p}(\Omega)$. 
Moreover, in view of the fact 
$\bigcap_{p\in [1,9(m-1))} L^{p}(\Omega) \subset L^{m-1}(\Omega)$, 
Lemma \ref{lem;esti4} is more suitable than the previous result. 

\end{remark}

Then, thanks to this lemma, we shall see some time regularity properties of $(\nep+\ep)^\gamma$ with some $\gamma>0$. 

\begin{lem}\label{lem;esti5}
Assume that \eqref{condi;case1} or \eqref{condi;case2} holds. 
Then for all $T>0$ there exists $C(T)>0$ such that 
\[
 \int_0^T \|\partial_t (\Ne+\ep)^\gamma \|_{(W^{2,4}_0(\Omega))^\ast}\le C(T) \quad \mbox{for all} \ \ep\in (0,1),
\]
where 
\[
\gamma := \begin{cases}
 1 & (0 < m\le 2), 
\\
 \frac m2 & (m >2). 
\end{cases}
\]
\end{lem}
\begin{proof}
Let $T>0$ and let $\psi \in W^{2,4}_0(\Omega)$ satisfy $\|\psi\|_{W^{2,4}(\Omega)} \le 1$. Here we note from the continuous embedding 
$W^{2,4}_0(\Omega) \hookrightarrow W^{1,\infty}(\Omega)$ that 
there exists $C_1>0$ such that $\|\psi\|_{W^{1,\infty}(\Omega)} \le C_1$. 
Now straightforward calculations and integration by parts 
enable us to have that  
\begin{align*}
 \io (\partial_t (\Ne+\ep)^\gamma)\psi 
 &= 
 -\gamma (\gamma -1) \io (\Ne + \ep)^{\gamma-2} \de(\Ne) |\na \Ne|^2\psi
 \\
&\quad\, 
 -\gamma \io (\Ne+\ep)^{\gamma -1} \de(\Ne) \na \Ne\cdot\na \psi 
\\
&\quad \,
 + \gamma (\gamma -1) \io \frac{(\Ne+\ep)^{\gamma -2}\Ne\chi(\Ce)}{1+\ep \Ne}(\na \Ne\cdot \na \Ce) \psi
\\ &\quad\,
 + \gamma \io \frac{(\Ne+\ep)^{\gamma-1}\Ne \chi(\Ce)}{1+\ep \Ne} \na \Ce\cdot \na \psi 
\\ & \quad \, 
 - \gamma \io  (\Ne+ \ep)^{\gamma -1} (\Ue\cdot\na \Ne) \psi
\\ 
& =: I_1 + I_2 + I_3 + I_4+ I_5.  
\end{align*}
We first deal with the case that $0 < m \le 2$, 
which means $\gamma =1$. Then we note that $I_1 = I_3=0$. 
Furthermore, we use the H\"older inequality to obtain that 
\[
 |I_2| \le C_2 \left( \io (\de(\Ne)|\na \Ne|)^{p_2} +1 \right)
\]
and 
\[
 |I_4| \le C_3 \kl{\io (\Ne+\ep)^\frac 43 + \io |\na \Ce|^4}
\]
as well as 
\[
 |I_5| = \gamma\left| \io \Ne\Ue \cdot \na \psi  \right| \le C_1 \gamma \io |\Ne \Ue|
\]
with some $C_2,C_3>0$, 
which with the standard duality argument implies from Lemmas \ref{lem;esti2-4} and \ref{lem;esti3} and the fact $p_1 > \frac{4}{3}$ that 
\begin{align*}
 \int_0^T \| \partial_t (\Ne+\ep)\|_{(W^{2,4}_0(\Omega))^\ast} 
 &= \int_0^T \sup \left\{\left|\io \partial_t (\Ne+\ep)\psi \right| \ \Big|\ \psi \in W^{2,4}_0(\Omega), \ \|\psi\|_{W^{2,4}(\Omega) }\le 1 \right\}
 \\
 &\le C_4 \iio \kl{(\de(\Ne)|\na \Ne|)^{p_2} +1 + (\Ne+\ep)^{p_1} + |\na \Ce|^4 + |\nep\uep|}
 \\
 &\le C_5(T)
\end{align*}
with some $C_4>0$ and $C_5(T)>0$, where $p_1,p_2$ are constants defined as \eqref{def;p1p2}. 
On the other hand, in the case that $m>2$, 
which implies $\gamma =\frac m2$, 
we obtain from \eqref{condi;D} and the Young inequality that 
\begin{align*}
&|I_1| \le C_6 \kl{\io |\na (\Ne + \ep)^\frac{m}{2}|^2 +\io |\na (\Ne+ \ep)^{m-1}|^2}, 
\\[1mm] 
 & |I_2| \le C_7 \kl{\io (\Ne+\ep)^m+ \io |\na (\Ne + \ep)^{m-1}|^2 }
\end{align*}
and 
\begin{align*}
 &|I_3| \le C_8 \kl{ \io |\na (\Ne + \ep)^\frac m2|^2 +\io |\na \Ce|^2}, 
\\[1mm] 
 & |I_4| \le C_9 \kl{ \io (\Ne + \ep)^m+\io |\na \Ce|^2}
\end{align*}
as well as 
\begin{align*}
 |I_5| \le C_{10}\io \kl{|\na (\Ne+\ep)^\frac m2|^2 + \io |\Ue|^2 }
\end{align*}
with some $C_6,C_7,C_8,C_9,C_{10}>0$. 
Then a combination of Lemmas \ref{lem;esti2-1}, 
\ref{lem;esti2-2}, \ref{lem;esti4} and 
the standard duality argument derives that 
\begin{align*}
 \int_0^T \| \partial_t (\Ne+\ep)^{\frac m2}\|_{(W^{2,4}_0(\Omega))^\ast} 
 \le C_{11}(T)
\end{align*}
holds with some $C_{11}(T)>0$. 
\end{proof}

Similar arguments in the proof of \cite[Lemma 5.3]{zhang_li} 
can derive the following lemma; 
thus we only introduce the statement. 
\begin{lem}\label{lem;esti6}
Assume that \eqref{condi;case1} or \eqref{condi;case2} holds. 
Then for all $T>0$ there exists $C(T)>0$ such that 
 \[
 \int_0^T \|\partial_t \sqrt{\Ce}\|_{(W^{2,4}_{0}(\Omega))^\ast} 
 + \int_0^T \|\partial_t \Ue\|_{(W^{1,\frac 52}_{0,\sigma}(\Omega))^\ast}^{\frac 54}\le C(T)\quad \mbox{for all} \ \ep\in(0,1). 
 \]
\end{lem}

\section{Convergences: Proof of Theorem \ref{mainthm1}}\label{sec5}

Before stating convergences properties, we define weak solutions of \eqref{P}. 
\begin{df}\label{def;weaksol}
A  triplet $(n, c, u)$ is called 
a  {\it global weak solution} of \eqref{P} if 
$n,c,u$ satisfy 
\begin{align*}
n\in L^1_{\rm loc}(\ol{\Omega}\times [0,\infty)),
\quad c\in L^1_{\rm loc}([0,\infty);W^{1,1}(\Omega)), 
\quad u\in L^1_{\rm loc}([0,\infty);W^{1,1}_0(\Omega))
\end{align*}
and 
\begin{align*}
 \int_0^n D(s) \, ds \in L^1_{\rm loc} ([0,\infty);W^{1,1}(\Omega)), 
\quad 
 \mu n^\alpha \in L^1_{\rm loc}(\ol{\Omega}\times [0,\infty))
\end{align*}
as well as
\begin{align*}
 nu, cu,n\chi(c) \na c, u\otimes u\in L^1_{\rm loc}(\ol{\Omega}\times [0,\infty))
\end{align*}
and the identities
\begin{align*}
    &-\int^\infty_0\!\!\!\!\int_\Omega n\vp_t
     -\int_\Omega n_0\vp(\cdot,0)
     -\int^\infty_0\!\!\!\!\int_\Omega n u\cdot\na\vp
   \\
    &\h{7.0mm}=-\int^\infty_0\!\!\!\!\int_\Omega\na \kl{\int_0^n D(s)\,ds} \cdot\na\vp
      +\int^\infty_0\!\!\!\!\int_\Omega n\chi(c) \na c\cdot\na\vp +\int^\infty_0\!\!\!\!\int_\Omega (\kappa n-\mu n^\alpha) \vp,
   \\[1.5mm]
    &-\int^\infty_0\!\!\!\!\int_\Omega c\vp_t
     -\int_\Omega c_0\vp(\cdot,0)
     -\int^\infty_0\!\!\!\!\int_\Omega cu\cdot\na\vp
   =-\int^\infty_0\!\!\!\!\int_\Omega\na c\cdot\na\vp
      -\int^\infty_0\!\!\!\!\int_\Omega nf(c)\vp,
   \\[1.5mm]
    &-\int^\infty_0\!\!\!\!\int_\Omega u\cdot\psi_t
     -\int_\Omega u_0\cdot\psi(\cdot,0)
     -\int^\infty_0\!\!\!\!\int_\Omega u\otimes u\cdot\na\psi
   \\
    &\h{7.0mm}=-\int^\infty_0\!\!\!\!\int_\Omega\na u\cdot\na\psi
     +\int^\infty_0\!\!\!\!\int_\Omega n\na\Phi \cdot \psi 
     + \int^\infty_0\!\!\!\!\int_\Omega  g\cdot \psi 
  \end{align*}
  hold for all $\vp\in C^{\infty}_0(\overline{\Omega}\times [0,\infty))$
  and all $\psi\in C^{\infty}_{0,\sigma}(\Omega \times [0,\infty))$,
  respectively.
\end{df}

Collecting the boundedness properties obtained in Lemmas \ref{lem;esti1}, \ref{lem;esti2-4}, \ref{lem;esti3}, \ref{lem;esti5} and \ref{lem;esti6}, we establish the following convergences.  

\begin{lem}\label{lem;conv1}
Assume that \eqref{condi;case1} or \eqref{condi;case2} holds. 
Then there exist a subsequence $\ep_j \searrow 0$ 
and functions $n,c,u$ such that 
for all $p\in [1,p_1)$ and $q\in [1,\infty)$, 
\begin{align*}
  \Ne &\to n && \mbox{in} \ L^p_{\rm loc}(\overline{\Omega}\times [0,\infty))  \ \mbox{and a.e.\ in} \ \Omega\times (0,\infty),
\\
 \nep^\alpha & \rightharpoonup n^\alpha && \mbox{in} \ L^1_{\rm loc}(\ol{\Omega}\times [0,\infty)),
\\
  \ep \nep^2 & \rightharpoonup 0 && \mbox{in} \ L^1_{\rm loc} (\ol{\Omega}\times [0,\infty)),   
\\
 \Ce &\to c && \mbox{in} \ L^q_{\rm loc}(\overline{\Omega}\times [0,\infty))\ \mbox{and a.e.\ in} \ \Omega\times (0,\infty),
 \\
 \Ce &\overset{\ast}{\rightharpoonup} c &&\mbox{in} \ L^\infty (\Omega\times (0,\infty)),
\\
 \nabla \Ce &\rightharpoonup \na c && \mbox{in} \ L^4_{\rm loc}(\overline{\Omega}\times [0,\infty)),
 \\ 
 \nabla \Ce^\frac{1}{4} &\rightharpoonup \na c^\frac{1}{4} 
 && \mbox{in} \ L^4_{\rm loc}(\overline{\Omega}\times [0,\infty)), 
 \\ 
 \Ue &\to u &&\mbox{in} \ L^2_{\rm loc}(\overline{\Omega}\times [0,\infty))  \ \mbox{and a.e.\ in} \ \Omega\times (0,T),
 \\ 
 \na \Ue & \rightharpoonup \na u && \mbox{in} \ L^2_{\rm loc} (\overline{\Omega}\times [0,\infty))
\end{align*}
hold as $\ep = \ep_j \searrow 0$, 
where $p_1>\frac 43$ is the constant defined in \eqref{def;p1p2}.
\end{lem}
\begin{proof}
This proof is based on \cite[Proof of Theorem 1.1]{zhang_li} and \cite[Proposition 6.1]{lankeit_evsmoothness}. 
Let $T>0$ and let  
\[
 \beta := 
  \begin{cases}
   p_2 & (0 < m \le 2),
  \\
   2 & (m>2)
  \end{cases}
  \quad \mbox{and} \quad  
  \gamma := 
   \begin{cases}
    1 & (0 < m \le 2),
   \\
    \frac m2 & (m>2),
   \end{cases}
\]
where $p_2$ is the constant defined in \eqref{def;p1p2}. 
Then since  Lemmas \ref{lem;esti2-4} and \ref{lem;esti5} 
derive that
\[
\left( (\Ne + \ep)^\gamma \right)_{\ep \in (0,1)} 
\ \mbox{is bounded in} \  L^\beta (0,T; W^{1,\beta}(\Omega)) 
\]
and 
\[ 
\left(
 \partial_t (\Ne+\ep)^\gamma )_{\ep\in(0,1)} \ \mbox{is bounded in}\ 
 L^1(0,T; (W^{2,4}_0(\Omega))^\ast 
\right)
\]
hold, a combination of the compact embedding 
$W^{1,\beta}(\Omega) \hookrightarrow L^\beta(\Omega) $ 
and the continuous embedding $L^\beta (\Omega) \hookrightarrow (W^{2,4}_0(\Omega))^\ast$ together with an Aubin--Lions-type lemma (see \cite[Corollary 4]{Simon_1987}) enables us to find a subsequence $\ep_j \searrow 0$ 
and a function $n$ such that 
\[
 (\Ne + \ep)^\gamma \to n^\gamma \quad \mbox{in} \ L^\beta (\Omega\times (0,T))\ \mbox{and a.e\ in} \  \Omega\times (0,T)  
\]
as $\ep =\ep_j\searrow 0$. 
Therefore, aided by Lemma \ref{lem;esti2-4}, 
we obtain from the Vitali convergence theorem that 
\[
\Ne \to n  \quad \mbox{in} \ L^p(\Omega\times (0,T)) \ \mbox{and a.e.\ in} \ \Omega\times (0,\infty)
\]
as $\ep=\ep_j\searrow 0$ for all $p\in [1,p_1)$,
where we used the relation $r > \frac{6}{5}$ ($r$ is the constant obtained in Lemma \ref{lem;esti3}). 
Moreover, since Lemma \ref{lem;esti1} yields that $(\nep^\alpha)_{\ep\in (0,1)}$ and $(\ep \nep^2)_{\ep\in (0,1)}$ are 
weakly relatively precompact by the Dunford--Pettis theorem, we can find a further subsequence (again denoted by $\ep_j$) and functions $z_1,z_2\in L^1(\Omega\times (0,T))$ such that 
\[
\nep^\alpha \rightharpoonup z_1, 
 \quad 
\ep \nep^2 \rightharpoonup z_2 
 \quad 
\mbox{in}\ L^1(\Omega\times (0,T)) 
\]
as $\ep=\ep_j\searrow 0$, which with pointwise a.e.\ 
convergences implies that $z_1=n^\alpha$ and $z_2=0$ (for more details, see the proof of \cite[Proposition 6.1]{lankeit_evsmoothness}).  
On the other hand, noticing from Lemmas \ref{lem;esti1} and \ref{lem;esti6} 
that $(\sqrt{\Ce})_{\ep\in (0,1)}$ and $(\partial_t \sqrt{\Ce})_{\ep\in(0,1)}$ are bounded in $L^2(0,T;W^{2,2}(\Omega))$ and 
$L^{1}(0,T; (W^{2,4}_{0}(\Omega))^\ast)$, respectively, 
as well as $(\Ue)_{\ep\in (0,1)}$ and $(\partial_t \Ue)_{\ep\in (0,1)}$ 
are bounded in $L^2(0,T;W_{0,\sigma}^{1,2}(\Omega))$ and $L^\frac{5}{4}(0,T;(W_{0,\sigma}^{1,\frac 52}(\Omega)^\ast))$, respectively, 
we obtain from the Aubin--Lions-type lemma that 
there are a further subsequence (again denoted by $\ep_j$) and 
functions $c, u$ such  that 
\begin{align*}
 \sqrt{\Ce} \to \sqrt{c} \quad \mbox{in} \ L^2(0,T;W^{1,2}(\Omega))
 \ \mbox{and a.e.\ in}\ \Omega\times (0,\infty)
\end{align*}
and 
\begin{align*}
 \Ue \to u \quad \mbox{in} \ L^2(\Omega\times(0,T))\ \mbox{and a.e.\ in} \ \Omega\times (0,\infty). 
\end{align*}
These convergences and Lemmas \ref{lem;basic}, \ref{lem;esti1} 
and \ref{lem;esti2-1} 
entail this lemma. 
\end{proof}

Then, aided by Lemma \ref{lem;conv1}, we can obtain the following convergences which are needed for the proof of Theorem \ref{mainthm1}. 

\begin{lem}\label{lem;conv2}
Assume that \eqref{condi;case1} or \eqref{condi;case2} holds and 
let $\ep_j,n,c,u$ be obtained in Lemma \ref{lem;conv1}. 
Then 
\begin{align*}
\Ne\Ue  \rightharpoonup nu, 
\quad 
\Ce\Ue  \rightharpoonup cu, 
\quad 
 \frac{f(\Ce)}{\ep} \log (1+\ep \Ne) & \to nf(c)
\quad \mbox{in} \ L^1_{\rm loc}(\overline{\Omega}\times [0,\infty))
\end{align*}
and 
\begin{align*}
Y_\ep \Ue \otimes \Ue \to u\otimes u, 
\quad \frac{\Ne}{1+\ep \Ne}\chi(\Ce) \na \Ce 
&\rightharpoonup n\chi(c)\na c
\quad \mbox{in} \ L^1_{\rm loc}(\overline{\Omega}\times [0,\infty)). 
\end{align*}
\end{lem}
\begin{proof}
Let $T>0$. 
 From Lemmas \ref{lem;esti3} and \ref{lem;conv1} we can see that 
 \begin{align*}
\Ne\Ue  \rightharpoonup nu, 
\quad 
\Ce\Ue  \rightharpoonup cu 
\quad \mbox{in} \ L^1(\Omega\times(0,T)) 
\end{align*}
and  
\[
 \frac{f(\Ce)}{\ep} \log (1+\ep \Ne)  \to nf(c) \quad 
 \mbox{a.e.\ in} \ \Omega\times(0,T)  
\]
hold as $\ep=\ep_j\searrow 0$. 
This together with the estimate  
\begin{align*}
 \iio \left| \frac{f(\Ce)}{\ep} \log (1+\ep \Ne)\right |^{p_1} 
 & \le \|f(s)\|_{L^\infty(0,\lp{\infty}{c_0})}^{p_1} \iio |\Ne|^{p_1} 
 \\
 &\le C_1(T)
\end{align*}
with some $C_1(T)>0$ (from Lemma \ref{lem;esti2-4}) implies from 
the Vitali convergence theorem that 
\[
 \frac{f(\Ce)}{\ep} \log (1+\ep \Ne) \to nf(c)
\quad \mbox{in} \ L^1_{\rm loc}(\overline{\Omega}\times [0,\infty))
\]
as $\ep=\ep_j \searrow 0$. Moreover, a combination of Lemma \ref{lem;conv1} and arguments in the proof of \cite[(4.26)]{Winkler_2016} yields that 
\[ 
Y_\ep \Ue \otimes \Ue \to u\otimes u \quad \mbox{in} \ L^1(\Omega\times(0,T)) 
\] 
as $\ep = \ep_j\searrow 0$. 
Now, in light of Lemmas \ref{lem;basic} and \ref{lem;conv1} together with the condition for $\chi$ (see \eqref{condi;chi}), 
the dominated convergence theorem implies that 
\[
 \frac{1}{1+\ep\Ne}\chi(\Ce) \Ce^{\frac{3}{4}} 
\to \chi(c)c^\frac 34 \quad \mbox{in} \ L^q(\Omega\times (0,T)) 
\]
for all $q\in [1,\infty)$. 
Therefore, noticing that 
$p_1 > \frac{4}{3}$, 
we obtain from Lemma \ref{lem;conv1} that 
\begin{align*}
 \frac{\Ne}{1+\ep \Ne} \chi(\Ce)\na \Ce & = \Ne \cdot \frac{4}{1+\ep\Ne}\chi(\Ce)\Ce^\frac 34 \cdot \na \Ce^\frac{1}{4} 
\\
 &\rightharpoonup 4n \cdot \chi(c)c^\frac 34 \cdot \na c^\frac 14 
 = n\chi(c)\na c 
 \quad \mbox{in} \ L^1(\Omega\times (0,T))
\end{align*}
as $\ep = \ep_j\searrow 0$. 
\end{proof}

We also desire the following lemma to show Theorem \ref{mainthm1}. 

\begin{lem}\label{lem;conv3}
Assume that \eqref{condi;case1} or \eqref{condi;case2} holds. 
Then $\int_0^n D(\sigma)\,d\sigma \in L^1_{\rm loc}([0,\infty);W^{1,1}(\Omega))$ and there is a further subsequence \/\,{\rm (}again denoted by $\ep_j${\rm )}\/\, such that 
\begin{align*}
\na \left( \int_{-\ep}^{\Ne} \de(\sigma) \,d\sigma \right) \rightharpoonup 
\na \left(\int_0^n D(\sigma) \, d\sigma \right) 
\quad \mbox{in} \ L^{p_2}_{\rm loc}(\overline{\Omega}\times [0,\infty)) 
\end{align*}
as $\ep = \ep_j \searrow 0$, where $p_2\in (1,2)$ is the constant defined in \eqref{def;p1p2}. 
Moreover, if $1 \le m \le 2$, then 
\[
 \na \left( \int_0^n D(\sigma)\, d\sigma \right) = D(n)\na n.
\]
\end{lem}
\begin{proof}
Let $T>0$. 
From Lemma \ref{lem;esti2-4} we can see that 
there exist a further subsequence (again denoted by $\ep_j$) and a function $w$ such that 
 \[
  \na \left( \int_{-\ep}^{\Ne} \de(\sigma)\, d\sigma \right) 
  \rightharpoonup w 
  \quad \mbox{in} \ L^{p_2}(\Omega \times (0,T)) 
 \]
 as $\ep=\ep_j\searrow 0$. 
 In order to see $w = \na (\int_0^n D(\sigma)\, d\sigma )$  
 we shall show that $\int_{-\ep}^{\Ne} \de(\sigma)\, d\sigma \to \int_0^n D(\sigma)\, d\sigma $ in $L^{p_2}(\Omega\times (0,T))$. 
Since Lemma \ref{lem;conv1} asserts that $\Ne + \ep \to n$ a.e.\ in $\Omega\times (0,T)$, 
we first establish that 
\[
 \int_{-\ep}^{\Ne} \de(\sigma)\,d\sigma = \int_0^{\Ne+\ep} D(\sigma)\,d\sigma  \to \int_0^{n} D(\sigma)\, d\sigma \quad \mbox{a.e.\ in} \ \Omega\times (0,T).
\]
On the other hand, 
the condition for $D$ (see \eqref{condi;D}) and Lemma \ref{lem;esti2-4} derive that 
\begin{align*}
 \iio \left| \int_{-\ep}^{\Ne} \de(\sigma)\, d\sigma \right|^{\frac{p_1}{m}} 
& \le D_2^{\frac{p_1}{m}} 
 \iio \left( \int_{-\ep}^{\Ne} (\sigma + \ep)^{m-1}\, d\sigma \right)^{\frac{p_1}{m}} 
 \\
 &=  \left(\frac{D_2}{m} \right)^{\frac{p_1}{m}} 
 \iio  \left( ( \Ne+\ep )^m \right)^{\frac{p_1}{m}} 
 \le C_1(T)
\end{align*}
with some $C_1(T)>0$. 
Therefore the Vitali convergence theorem entails from the fact $\frac {p_1}m > p_2$ that 
\[
\int_{-\ep}^{\Ne} \de(\sigma)\, d\sigma \to \int_0^n D(\sigma)\, d\sigma  \quad \mbox{in} \ L^{p_2}(\Omega\times (0,T)),
\]
which means that $w$ coincides with $\na (\int_0^n D(\sigma)\, d\sigma )$. 
Moreover, if $1 \le m \le 2$, then 
from Lemma \ref{lem;esti2-2} and the Vitali convergence theorem 
we can find a further subsequence (again denoted by $\ep_j$) such that 
\[
 \de(\Ne) \to D(n) \quad \mbox{in} \  L^{\frac{3m+2}{3m-2}}(\Omega\times (0,T)) 
\]
and 
\[
 \na \Ne \rightharpoonup \na n \quad \mbox{in} \ L^{\frac{3m+2}{4}}(\Omega\times (0,T))
\]
as $\ep = \ep_j \searrow 0$. 
Thus we see that 
\[
 \na \left( \int_{-\ep}^{\Ne} \de(\sigma) \,d\sigma \right) 
  = \de(\Ne) \na \Ne \rightharpoonup D(n) \na n 
  \quad \mbox{in} \ L^1(\Omega\times (0,T)), 
\]
which means that $\na (\int_0^n D(\sigma)\, d\sigma) = D(n) \na n$ holds. 
\end{proof}

Thanks to convergences established in Lemmas \ref{lem;conv1}, 
\ref{lem;conv2} and \ref{lem;conv3}, we can establish existence of a global weak solution in the sense of Definition \ref{def;weaksol}. 

\begin{proof}[{\bf Proof of Theorem \ref{mainthm1}}]
Testing the first and second equations in \eqref{Pe} by $\vp\in C^{\infty}_0(\overline{\Omega}\times [0,\infty))$
and testing the third equation in \eqref{Pe} by $\psi\in C^{\infty}_{0,\sigma}(\Omega \times [0,\infty))$, we have that 
\begin{align*}
    -\int^\infty_0\!\!\!\!\int_\Omega \nep \vp_t &
     -\int_\Omega n_{0\ep}\vp(\cdot,0)
     -\int^\infty_0\!\!\!\!\int_\Omega \nep \uep \cdot\na\vp
   \\
    &
    =-\int^\infty_0\!\!\!\!\int_\Omega\na \kl{\int_0^{\nep} \de(s)\,ds} \cdot\na\vp
      +\int^\infty_0\!\!\!\!\int_\Omega \frac{\nep\chi(\cep)}{1+\ep \nep} \na \cep \cdot\na\vp%
   \\ & \quad\, 
      + \int^\infty_0\!\!\!\!\int_\Omega (\kappa \nep - \mu \nep^\alpha - \ep \nep^2)\vp
\end{align*}
and 
\begin{align*}
    -\int^\infty_0\!\!\!\!\int_\Omega \cep \vp_t &
     -\int_\Omega c_{0\ep}\vp(\cdot,0)
     -\int^\infty_0\!\!\!\!\int_\Omega \cep \uep \cdot\na\vp
   \\
    & =-\int^\infty_0\!\!\!\!\int_\Omega\na \cep \cdot\na\vp
      -\frac 1\ep \int^\infty_0\!\!\!\!\int_\Omega f(\cep)\log (1+\ep \nep) \vp 
\end{align*}
as well as
\begin{align*}
    -\int^\infty_0\!\!\!\!\int_\Omega \uep \cdot\psi_t & 
     -\int_\Omega u_{0\ep}\cdot\psi(\cdot,0)
     -\int^\infty_0\!\!\!\!\int_\Omega Y_\e \uep \otimes \uep \cdot\na\psi
   \\
    & 
    = -\int^\infty_0\!\!\!\!\int_\Omega\na \uep \cdot\na\psi
     +\int^\infty_0\!\!\!\!\int_\Omega \nep \na\Phi \cdot \psi 
     + \int^\infty_0\!\!\!\!\int_\Omega  g_\e\cdot \psi 
  \end{align*}
for all $\ep\in (0,1)$. 
Now we note that 
the condition \eqref{condi;alpha,m} implies that 
\eqref{condi;case1} or \eqref{condi;case2} holds. 
Thus, plugging \eqref{ini;app;nep}--\eqref{ini;app;gep} and Lemmas \ref{lem;conv1}, \ref{lem;conv2}, \ref{lem;conv3} into the above identities, by taking the limit as $\ep=\ep_j\searrow 0$ we can obtain Theorem \ref{mainthm1}. 
\end{proof}

\begin{remark}
This paper shows existence of global weak solutions under some largeness condition for $m$ and $\alpha$. However, because of difficulties of the nonlinear diffusion, we could not deal with large time behavior of these solutions; the large time behavior of solutions to \eqref{P} is still an open problem. 
In view of the results in \cite{Lankeit_2016,win_transAMS} we can expect that 
solutions become smooth after some time and converge to the steady state. 
\end{remark}



\end{document}